\documentclass{article}
\usepackage{amsmath} 
\usepackage{amsthm} 
\usepackage{mathtools} 
\usepackage{amssymb} 
\usepackage[left=1.5in,right=1.5in,top=1in,bottom=1in]{geometry}
\usepackage{hyperref} 
\usepackage{cleveref} 
\usepackage{mathrsfs} 
\usepackage{tikz-cd} 
\usepackage{tikz} 
\usepackage{import} 
\usepackage{enumitem} 
\usepackage[mathscr]{eucal}

\DeclareMathOperator{\out}{Out}
\DeclareMathOperator{\Mod}{Mod}
\DeclareMathOperator{\st}{st}
\DeclareMathOperator{\pf}{PF}

\newcommand{\pfmin}{\pf_{\operatorname{min}}}

\newtheorem{THM}{Theorem}
\newtheorem{thm}{Theorem}
\numberwithin{thm}{section}
\newtheorem{lem}[thm]{Lemma}

\newtheorem{prop}[thm]{Proposition}
\newtheorem{cor}[thm]{Corollary}
\theoremstyle{definition}
\newtheorem{ex}[thm]{Example}
\newtheorem{rk}[thm]{Remark}

\usetikzlibrary{decorations.markings, shapes, quotes}
\tikzset{->-/.style = {
		very thick,
		decoration = {
			markings,
			mark = at position .5 with {\arrow{>}}
		},
		postaction = {decorate}
	},
	pt/.style = {
		circle,
		fill = black,
		scale = 0.5,
	},
}

\begin{document}
\title{Train track maps on graphs of groups}
\author{Rylee Alanza Lyman}
\maketitle

\begin{abstract}
    In this paper we develop the theory of train track  maps on graphs of groups.
    Expanding a definition of Bass,
    we define a notion of a map of a graph of groups,
    and of a homotopy equivalence.
    We prove that under one of two technical hypotheses,
    any homotopy equivalence of a graph of groups
    may be represented by a relative train track map.
    The first applies in particular to graphs of groups with finite edge groups,
    while the second applies in particular to certain generalized Baumslag--Solitar groups.
\end{abstract}

A homotopy equivalence $f\colon G \to G$ of a connected graph $G$
is a \emph{train track map}
when $f$ maps vertices to vertices
and the restriction of any iterate of $f$
to an edge of $G$ yields an immersion.
(Relative) train track maps were introduced in \cite{BestvinaHandel};
they are perhaps the main tool for studying outer automorphisms of free groups.

A train track map $f\colon G \to G$ 
induces a well-defined outer automorphism of $\pi_1(G)$, a free group.
The train track condition  simplifies the analysis of the action of $f$
on paths and loops in $G$.
Choosing a basepoint $\star$ in $G$ and a path from $\star$ to $f(\star)$
determines an automorphism $f_\sharp\colon \pi_1(G,\star) \to \pi_1(G,\star)$
and a lift of $f$ to the universal covering tree $\Gamma$ of $G$.
The map $\tilde f\colon \Gamma \to \Gamma$
is \emph{$f_\sharp$-twisted equivariant} in the sense that
for $g \in \pi_1(G,\star)$ and $x \in\Gamma$, we have
\[  \tilde f(g.x) = f_\sharp(g).\tilde f(x). \]
The lift $\tilde f\colon \Gamma \to \Gamma$ also satisfies
the definition of a train track map.
This formulation of train track maps as twisted equivariant maps of trees
can be adapted in a straightforward way to automorphisms of groups acting on trees.

\begin{THM}[Tree version]
    Suppose a group $F$ acts cocompactly on a simplicial tree $T$,
    that $\Phi\colon F \to F$ is an automorphism,
    and that $\tilde f\colon T \to T$ is a $\Phi$-twisted equivariant map.
    Assuming one of the following conditions holds,
    there exists a $\Phi'$-twisted equivariant relative train track map
    $\tilde f'\colon T' \to T'$,
    where $\Phi$ and $\Phi'$ represent the same outer automorphism $\varphi$
    and where $T$ and $T'$ belong to the same deformation space $\mathscr{D}$
    in the sense of Forester \cite{Forester,GuirardelLevittDeformation}.
    \begin{enumerate}
        \item Let $T''$ be a reduced tree in $\mathscr{D}$.
            Assume that edge stabilizers of $T''$ are finitely generated
            and that no iterate of $\Phi$
            maps a generalized edge group of $T''$
            properly into a conjugate of itself.
        \item Assume that $F$ is finitely generated, $T$ is locally finite,
            and that the subgroup $\Mod(\mathscr{D})$ of $\out(F)$ leaving $\mathscr{D}$ invariant
            acts with finitely many orbits of cells
            on the deformation retract $P\mathscr{G}$ of $\mathscr{D}$
            considered in \cite[Theorem 7.6]{GuirardelLevittDeformation}
            \cite{Clay}.
    \end{enumerate}
    If $\varphi$ is irreducible, then the relative train track map constructed is a train track map.
\end{THM}

Two $F$-trees $T$ and $T'$ \emph{belong to the same deformation space} if and only if
\cite[Theorem 3.8]{GuirardelLevittDeformation}
there exist $F$-equivariant maps $T \to T'$ and $T' \to T$.
A tree $T$ is \emph{reduced} if collapsing any orbit of edges of $T$ yields a tree
not in the same deformation space.
A \emph{generalized edge stabilizer} is a subgroup $H \le F$
with the property that $H$ contains the stabilizer of some edge $\tilde e$ of $T$
and is contained in the stabilizer of another edge $\tilde e'$ of $T$.
Each relative train track map $\tilde f\colon T\to T$ is a \emph{morphism:}
after $F$-equivariantly subdividing edges in the domain tree into finitely many edges,
the map $\tilde f$ becomes \emph{simplicial,}
in the sense that it maps edges to edges.
Twisted equivariant morphisms of trees map edge stabilizers to generalized edge stabilizers.

Relative train track maps are defined in \Cref{relativetraintracksection}.
This is not the first construction of relative train track maps
on graphs of groups,
see \cite{CollinsTurner}, \cite{FrancavigliaMartino} and \cite{Sykiotis},
but it is the most general,
allowing in particular for infinite edge groups
(but see \cite{Meinert} in the irreducible case).
The proof of \Cref{relativetraintrack}
relies on an algorithm of Bestvina--Handel \cite{BestvinaHandel},
which requires a bound on the number of edges of $G$,
the underlying graph of $\mathcal{G}$.
Without further assumptions on $\mathcal{G}$,
it appears at least possible that certain problematic  valence-two
vertices could proliferate in $G$,
destroying any guarantee that the algorithm will terminate.
It would be very interesting to have an example where this proliferation
actually occurs.
Perhaps an example could be found in considering the group
\[  \langle a, s, t : tat^{-1} = sas^{-1} = a^2 \rangle. \]

Since one is primarily interested in using train track maps
to study \emph{outer} automorphisms,
the choice of automorphism $\Phi$ in the statement of \Cref{relativetraintrack} is inconvenient.
It would be more convenient
to be able to work directly in the quotient graph of groups.
This is the purpose of this paper.

Bass \cite{Bass} defines a notion of a morphism of a graph of groups
and proves that his morphisms of graphs of groups induce twisted equivariant simplicial maps of trees
and vice versa.
In \Cref{morphisms},
we offer an expanded definition of a \emph{map} of a graph of groups
and prove that our maps induce twisted equivariant maps of trees sending vertices to vertices
and vice versa.
We define \emph{homotopy} of maps and when a map is a \emph{homotopy equivalence.}
The Bass--Serre trees of homotopy equivalent graphs of groups belong to the same deformation space
and conversely.

\setcounter{THM}{0}
\begin{THM}[Graph of groups version]
    \label{relativetraintrack}
    Let $\mathcal{G}$ be a finite, connected graph of groups,
    $\varphi$ be an outer automorphism of $\pi_1(\mathcal{G})$,
    and suppose that $\varphi$ is induced by a map $f\colon \mathcal{G} \to \mathcal{G}$
    satisfying one of the following conditions.
    Then there exists a relative train track map $f'\colon \mathcal{G}' \to \mathcal{G}'$
    representing $\varphi$ on a graph of groups $\mathcal{G}'$
    homotopy equivalent to $\mathcal{G}$.
    \begin{enumerate}
        \item Let $\mathcal{G}''$ be a reduced graph of groups homotopy equivalent to $\mathcal{G}$.
            Assume that edge groups of $\mathcal{G}''$ are finitely generated 
            and for some and hence every $\Phi$ representing $\varphi$,
            no iterate of the map $\Phi$ induces a proper inclusion
            of a generalized edge group of $\mathcal{G}$ into itself.
        \item Assume that vertex groups of $\mathcal{G}$ are finitely generated
            and edge groups of $\mathcal{G}$ have finite index
            in their incident vertex groups.
            Assume further that 
            there are only finitely many isomorphism types of graphs of groups $\mathcal{G}'$
            homotopy equivalent to $\mathcal{G}$ with the property
            that each edge $e$ of $\mathcal{G}'$ is \emph{surviving,}
            in the sense that there is some reduced collapse of $\mathcal{G}'$
            in which the edge $e$ is not collapsed.
    \end{enumerate}
    If $\varphi$ is irreducible,
    then the relative train track map constructed is a train track map.
\end{THM}

Item 1 of \Cref{relativetraintrack} applies in particular
whenever edge groups are finite, 
(or more generally when generalized edge groups are co-Hopfian)
and thus to all accessible groups with infinitely many ends.
Item 2 applies in particular to certain generalized Baumslag--Solitar groups.

Part of this work originally appeared in the author's thesis \cite{MyThesis}.
The author would like to thank Lee Mosher for many helpful conversations
and comments on early drafts of this work,
Chlo\'e Papin for conversations which led to realizing 
that further assumptions were necessary,
Mark Feighn and Mladen Bestvina for suggestions on alternate assumptions
and the anonymous referee for numerous comments
and a very careful reading
which helped improve the exposition of this article.

The strategy of the proofs in this paper
is to find the correct equivariant perspective
so that the original arguments in \cite{BestvinaHandel} 
and \cite{FeighnHandelAlg}
can be adapted without too much  extra effort.

Here is the organization of the paper.
We build up the aforementioned equivariant perspective in \Cref{morphisms}.
The proof of \Cref{relativetraintrack}
follows the outline in \cite{BestvinaHandel};
it occupies \Cref{traintracksection} and \Cref{relativetraintracksection} here.

\section{Maps of graphs of groups}
\label{morphisms}
The purpose  of this section is to define \emph{maps} of graphs of groups
and discuss their relationship with twisted equivariant maps of trees.
\emph{Morphisms} of graphs of groups were originally defined by Bass \cite{Bass}.
Our definition differs from his in two main respects:
first, while his morphisms send edges to edges,
our maps may send edges to edge paths (which may contain no edges),
and second, we require our maps to respect basepoints.
Assuming a map does not collapse edges to vertices,
we call it a morphism, and 
it becomes a morphism in the sense of Bass after subdividing edges in the domain graph of groups
into finitely many edges.

Let us take up the discussion from the introduction.
Let $F$ and $F'$ be groups acting on simplicial trees $\Gamma$ and $\Gamma'$,
let $\Phi\colon F \to F'$ be a homomorphism,
and suppose there is a (continuous) map $\tilde f\colon \Gamma \to \Gamma'$
which is $\Phi$-twisted equivariant in the sense that for all $\tilde x\in \Gamma$
and all $g \in F$, we have
\[  \tilde f(g.\tilde x) = \Phi(g).\tilde f(\tilde x). \]
An \emph{equivariant homotopy} between two $\Phi$-twisted equivariant maps $\tilde f$ and $\tilde f'$
is a homotopy $\tilde f_t\colon \Gamma \to \Gamma'$
with  $\tilde f_0  = \tilde f$ and $\tilde f_1 = \tilde f'$
such that each map $\tilde f_t$ is $\Phi$-twisted equivariant.
A $\Phi$-twisted equivariant map $\tilde f\colon \Gamma \to \Gamma'$ is a \emph{homotopy equivalence}
if $\Phi$ is an isomorphism and there exists
a $\Phi^{-1}$-twisted equivariant map $\tilde g \colon \Gamma'  \to \Gamma$
such that $\tilde f\tilde g$ and $\tilde g\tilde f$ are equivariantly homotopic to the identity.
If we use $\Phi$ to identify $F$ with $F'$,
this says that $\Gamma$ and $\Gamma'$ belong to the same \emph{deformation space}
in the sense of \cite{Forester} \cite{GuirardelLevittDeformation}.

Each $\Phi$-twisted equivariant map $\tilde f\colon \Gamma \to \Gamma'$
is equivariantly homotopic to a map $\tilde f'\colon \Gamma \to \Gamma'$
which sends vertices to vertices
and which has the property that for each edge $\tilde e$ of $\Gamma$,
either $\tilde f(\tilde e)$ is a vertex
or after subdividing $\tilde e$ into finitely many edges,
the map $\tilde f$ restricted to the newly created edges is simplicial.
We will work exclusively with such maps.
We say the image of $\tilde f(\tilde e)$ is an \emph{edge path}
$\tilde e'_1\ldots \tilde e'_k$ in $\Gamma'$.
In other words, if $\tilde f$ does not collapse edges,
it is a $\Phi$-twisted equivariant \emph{morphism} of trees.

\paragraph{Quotient graph of groups.}
A \emph{graph of groups} $\mathcal{G}$ is a graph $G$ (i.e.~a $1$-dimensional CW complex),
which we usually assume to be connected,
together with, for each edge $e$ and vertex $v$ of $G$,
an assignment of groups $\mathcal{G}_e$ and $\mathcal{G}_v$.
For an oriented edge $e$ with initial vertex $v$ and terminal vertex $w$,
there are injective homomorphisms
$\iota_{\bar e} \colon \mathcal{G}_e \to \mathcal{G}_v$
and  $\iota_e \colon \mathcal{G}_e \to \mathcal{G}_{w}$, respectively.
(We have $\mathcal{G}_{\bar e} = \mathcal{G}_e$.)
The reader is referred to \cite{Bass,Trees,ScottWall,MyThesis} 
for additional background on graphs of groups,
although we give a reasonably self-contained exposition of the aspects of the theory we will use.

Suppose $F$ is a group acting on a simplicial tree $\Gamma$
by simplicial automorphisms,
and suppose that the action is \emph{without inversions in edges,}
i.e.~that no group element sends an edge $\tilde e$ to itself reversing orientation.
This can always be arranged by passing to the barycentric subdivision of $\Gamma$.
There is a \emph{quotient graph of groups} $\mathcal{G}$, which we now describe.

Since the action of $F$ on $\Gamma$ is without inversions in edges,
the quotient $F\backslash\Gamma$ naturally inherits the structure of a graph
 from $\Gamma$, call this graph $G$.
The graph of groups structure on $G$ depends on a choice
of \emph{fundamental domain} for the action of $F$ on $\Gamma$;
we now describe how to choose the fundamental domain.
Choose a spanning tree $S \subset G$
and lift $S$ to $\tilde S \subset \Gamma$.
For each edge $e \notin S$,
choose a lift $\tilde e$ in $\Gamma$
such that one vertex of $\tilde e$ belongs to $\tilde S$.
Write $T$ for the union of $\tilde S$ with the (closed) edges $\tilde e$
for $e \notin S$.
For $v$ a vertex of $G$, set $\mathcal{G}_v$ to be the stabilizer of the unique preimage $\tilde v$
of $v$  in $\tilde S$.
For $e$ an edge of $G$, set $\mathcal{G}_e$ to be the stabilizer of the unique preimage $\tilde e$
of $e$ in $T$.
Let $e$ be an oriented edge with terminal vertex $v$
and write $\tilde x$ for the terminal vertex of $\tilde e$.
By definition, there is some group element $g_e \in F$
such that $g_e.\tilde x = \tilde v$.
If $\tilde x = \tilde v$, choose $g_e = 1 \in F$.
If $h \in F$ stabilizes $\tilde e$,
then $g_ehg_e^{-1}$ stabilizes $\tilde v$,
so define $\iota_e\colon \mathcal{G}_e \to \mathcal{G}_v$ to be the map
$h \mapsto g_ehg_e^{-1}$.
This defines the graph of groups structure $\mathcal{G}$ on the quotient graph $G$.

\paragraph{Graphs of spaces.}
Associated to a graph of groups $\mathcal{G}$ with underlying graph $G$,
we can build a \emph{graph of spaces} $X_\mathcal{G}$.
See \cite{ScottWall} for more details.
For a vertex $v$ of $G$, take a connected, CW complex $X_v$
with one vertex $\star_v$ which is a $K(\mathcal{G}_v,1)$
and fix an identification $\pi_1(X_v,\star_v) = \mathcal{G}_v$.
Do the same for each edge $e$ of $G$, producing a CW complex $X_e$
with one vertex $\star_e$ and $\pi_1(X_e,\star_e ) = \mathcal{G}_e$.
Suppose the oriented edge $e$ has initial vertex $v$ and terminal vertex $w$.
Associated to the homomorphisms $\iota_e\colon \mathcal{G}_e \to \mathcal{G}_{w}$
and $\iota_{\bar e}\colon \mathcal{G}_e \to \mathcal{G}_{v}$,
there are continuous, skeleta-preserving maps 
$i_e \colon (X_e,\star_e) \to (X_{w},\star_{w})$
and $i_{\bar e}\colon (X_e,\star_{e}) \to (X_{v},\star_{v})$
such that the induced maps on fundamental groups satisfy
$(i_e)_\sharp = \iota_e$ and $(i_{\bar e})_\sharp = \iota_{\bar e}$.
If $V$ is the set of vertices of $G$ and $E$ is the set of oriented edges,
the graph of spaces $X_\mathcal{G}$ is the quotient of the disjoint union
\[  \coprod_{v \in V} X_v \amalg \coprod_{e \in E} (X_e \times [0,1]) \]
by the equivalence relation identifying $(x, 1) \in X_e\times [0,1]$
with $i_e(x) \in X_{w}$,
where $w$ is the terminal vertex of the oriented edge $e$,
and identifying $(x,t) \in X_e\times [0,1]$ with $(x, 1-t) \in X_{\bar e}\times [0,1]$.
Thus $X_\mathcal{G}$ is a CW complex.
Note that after identifying each (open) edge of $G$ with $(0,1)$,
there is a surjection $X_\mathcal{G} \to G$
whose fibers are naturally identified with the spaces $X_v$ and $X_e$.
By identifying $G$ with the subspace of $X_\mathcal{G}$
comprising the points $\star_v$ and $(\star_e,t)$ for $t \in (0,1)$,
we can view the map $X_\mathcal{G} \to G$ as a retraction.

\paragraph{The fundamental group of a graph of groups.}
The \emph{fundamental group of the graph of groups $\pi_1(\mathcal{G})$}
is the fundamental group of the graph of spaces $X_\mathcal{G}$.
For convenience, choose a basepoint $p \in X_\mathcal{G}$ in the image of the retraction
$X_\mathcal{G} \to G$.
Each loop in $\pi_1(X_\mathcal{G},p)$ is homotopic into the $1$-skeleton of $X_\mathcal{G}$,
and thus may be represented as an \emph{edge path} $\gamma$ of the form
\[  \gamma = e'_1g_1e_2g_2\ldots e_kg_ke'_{k+1} \]
where $e_2,\dotsc,e_k$ are edges of $G$, $e'_1$ and $e'_{k+1}$
terminal and initial segments of edges $e_1$ and $e_{k+1}$ of $G$ respectively,
where the $g_i$ for $1 \le i \le k$ are elements of $\pi_1(X_{v_i},\star_{v_i}) = \mathcal{G}_{v_i}$,
and where $v_i = \tau(e_i) = \tau(\bar e_{i+1})$.
We allow the case where $e'_1$ and $e'_{k+1}$ are empty,
in which case they will be dropped from the notation.
A path is \emph{nontrivial} if it contains (a segment of) an edge.

Notice that under our identifications of $\pi_1(X_v,\star_v)$ with $\mathcal{G}_v$,
the notion of an edge path in $\mathcal{G}$ makes sense without reference to $X_\mathcal{G}$.
Homotopy rel endpoints of paths in $X_\mathcal{G}$
yields a corresponding notion of homotopy rel endpoints for edge paths in $\mathcal{G}$.
It is generated by replacing a segment of the form $e\iota_e(h)$ with $\iota_{\bar e}(h)e$,
where $e$ is an edge and $h \in \mathcal{G}_e$ is an element of the edge group,
and by adding or removing segments of the form $e\bar e$ for an edge $e$.
An edge path $\gamma$ is \emph{tight} if the number of edges in $\gamma$ 
cannot be lessened by a homotopy.

\paragraph{Maps of graphs of groups.}
A \emph{map} of graphs of groups
is a pair of maps $f\colon G \to G'$ and 
$f_X \colon X_\mathcal{G} \to X_{\mathcal{G}'}$
such that the following diagram commutes
\[  \begin{tikzcd}
    X_\mathcal{G} \ar[r, "f_X"] \ar[d, "r"] & X_{\mathcal{G}'} \ar[d, "r'"] \\
    G \ar[r, "f"] & G',
\end{tikzcd}\]
where $r$ and $r'$ are the retractions.
A \emph{homotopy} of maps is a pair of homotopies $f_{X,t}\colon X_{\mathcal{G}} \to X_{\mathcal{G}}$
and $f_t \colon G \to G$ 
such that for each $t$, the diagram of the form above commutes.
A map $f\colon \mathcal{G} \to \mathcal{G}'$
is a \emph{homotopy equivalence}
if there exists a map $g\colon \mathcal{G}' \to \mathcal{G}$ 
such that $fg$ and $gf$ are each homotopic  to the respective identity maps.
If $f$ is a homotopy equivalence, then the map $f_X$ induces an isomorphism of fundamental groups
$\pi_1(X_\mathcal{G}) \to \pi_1(X_{\mathcal{G}'})$,
but this is not a sufficient condition for $f$ to be a homotopy equivalence in general.
By the cellular approximation theorem,
every map $f\colon \mathcal{G} \to \mathcal{G}'$ of graphs of groups is homotopic 
to a map $f'\colon \mathcal{G} \to \mathcal{G}'$
with the property that the map $f'\colon G \to G'$
sends vertices to vertices
and either collapses edges to vertices or expands edges over edge paths
and the property that the map $f'_X \colon X_{\mathcal{G}} \to X_{\mathcal{G}'}$
sends the vertices $\star_v$ of $X_{\mathcal{G}}$ to vertices of $X_{\mathcal{G}'}$.
We will only consider such maps.

For such a map of graphs of groups, we turn now to collecting information
that will let us describe maps of graphs of groups 
without reference to $X_{\mathcal{G}}$ and $X_{\mathcal{G}'}$.
For each vertex $v$ of $G$, the map $f'_X$ induces a homomorphism
$f'_v\colon \mathcal{G}_v \to \mathcal{G}'_{f(v)}$,
and sends the oriented edge $e$ of $G$ 
(thought of as the subspace $\{\star_e\}\times[0,1]$ of $X_{\mathcal{G}}$)
to an edge path $f'(e) = g_0e'_1g_1\ldots e'_kg_k$ in $\mathcal{G}'$.
Notice that the edge path $e'_1\ldots e'_k$ in $G'$ is (homotopic to) the image of the edge $e$
under the map of graphs $f'\colon G \to G'$.

Suppose the edge $e$ has initial vertex $v$ and terminal vertex $w$ and that
the edge path $f'(e)$ is trivial, i.e.~$f'(e) = g_0$.
Then $f'(v) = f'(w)$,
and the images in $X_{\mathcal{G}'}$ of the $2$-cells of $X_{\mathcal{G}}$ 
recording the relations $\iota_{\bar e}(h)e\iota_{e}(h^{-1})\bar e$ for each $h \in \mathcal{G}_e$ 
imply that the following diagram commutes
\[  \begin{tikzcd}
    \mathcal{G}_e \ar[r, "\iota_e"] \ar[d, "\iota_{\bar e}"] 
    & \mathcal{G}_w \ar[d, "\operatorname{ad}(g_0)f'_w"] \\
    \mathcal{G}_v \ar[r, "f'_v"] & \mathcal{G}'_{f'(v)} = \mathcal{G}'_{f'(w)},
\end{tikzcd}    \]
where $\operatorname{ad}(g_0)$ is the inner automorphism $x \mapsto g_0 x g_0^{-1}$.
If the edge path $f'(e) = g_0e'_1g_1\ldots e'_kg_k$ is nontrivial,
then there are homomorphisms $f'_{e,e'_i} \colon \mathcal{G}_e \to \mathcal{G}'_{e'_i}$
for each $i$ satisfying $1 \le i \le k$,
and similarly we have the following commutative diagram
where $v_i$ is the terminal vertex of $e'_i$ and the initial vertex of $e'_{i+1}$
\[  \begin{tikzcd}
    \mathcal{G}_v \ar[d, "f'_v"] & & & \mathcal{G}_{e} \ar[lll, "\iota_{\bar e}"'] \ar[rrr, "\iota_e"]
    \ar[lld, "f'_{e,e'_i}"'] \ar[rrd, "f'_{e,e'_{i+1}}"] & & & \mathcal{G}_w 
    \ar[d, "\operatorname{ad}(g_k)f'_w"] \\
    \mathcal{G}'_{f'(v)} & \cdots \ar[l, "\operatorname{ad}(g_0)\iota_{\bar e'_1}"']
                         \mathcal{G}'_{e'_i} \ar[rr,"\iota_{e'_i}"]
                         & & \mathcal{G}'_{v'_i}  &
                         & \mathcal{G}'_{e'_{i+1}} \ar[ll, "\operatorname{ad}(g_i)\iota_{\bar e'_{i+1}}"'] 
                         \cdots \ar[r, "\iota_{e'_k}"] & \mathcal{G}'_{f'(w)}.
\end{tikzcd}\]
This diagram, coupled with the rule that $\iota_{\bar e'_i}(h)e'_i = e'_i\iota_{e'_i}(h)$
for $h \in \mathcal{G}_{e'_i}$
implies that we have
\begin{gather*}  
    f'_v\iota_{\bar e}(h)f'(e) = f'_v\iota_{\bar e}(h)g_0e'_1g_1\ldots e'_kg_k 
    = g_0 \iota_{\bar e'_1}f'_{e,e'_1}(h) e'_1 g_1 \ldots e'_kg_k  \\ =  \cdots = 
    g_0 e'_1 g_1 \ldots e'_k g_kf'_w\iota_e(h) = f'(e)f'_w\iota_{e}(h). 
\end{gather*}
Notice as well that because each homomorphism on the bottom row is injective,
the images of the maps $f'_{e,e'_i}$ are all abstractly isomorphic.
We will understand a \emph{map of graphs of groups} $f'\colon \mathcal{G} \to \mathcal{G}'$
to be the data of a map of graphs $f'\colon G \to G'$
which either collapses edges to vertices or maps edges to edge paths,
edge paths $f'(e) = g_0e'_1g_2\ldots e'_kg_k  \in \mathcal{G}'$,
and the homomorphisms $f'_v\colon \mathcal{G}_v \to \mathcal{G}'_{f(v)}$
and $f'_{e,e'_i} \colon\mathcal{G}_e \to \mathcal{G}_{e'_i}$
subject to the \emph{compatibility conditions} given by the above commutative diagrams.
A map of graphs of groups is a \emph{collapse map} if
the map of graphs $f\colon G \to G'$ either sends edges to edges or collapses edges to vertices.
It is a \emph{morphism} if the map of graphs  $f\colon G \to G'$ does not collapse edges.
If $f\colon \mathcal{G} \to \mathcal{G}$ is a homotopy equivalence,
a morphism,
and for each edge $e$ of $G$,
the edge path $f(e) = g_0e_1g_1\ldots e_kg_k$ is tight,
then we say $f$ is a \emph{topological representative}
of the induced outer automorphism $\varphi \in \out(\pi_1(\mathcal{G}))$.

Given a map of graphs of groups $f\colon \mathcal{G} \to \mathcal{G}'$,
the rule 
\[  g \in \mathcal{G}_v \mapsto f_v(g) \in \mathcal{G}'_{f(v)} \quad\text{and}\quad 
e \mapsto f(e) = g_0e'_1g_1  \ldots e'_kg_k \]
defines the action of $f$ on paths in $\mathcal{G}$ and a well-defined homomorphism 
$f_\sharp\colon \pi_1(X_{\mathcal{G}},\star_w) \to \pi_1(X_{\mathcal{G}'},\star_{f(w)})$.
Since $X_{\mathcal{G}}$ is a $K(\pi_1(X_{\mathcal{G}}),1)$,
the homomorphism $f_\sharp$ is induced by a map $f_X \colon X_{\mathcal{G}} \to X_{\mathcal{G}'}$.
We claim that we can choose $f_X$ and $f$ within their homotopy classes so that the pair $(f_X,f)$
defines a map of graphs of groups.
Since we have little further need of $X_{\mathcal{G}}$, we leave the details to the interested reader.

The following operations preserve the homomorphism
$f_\sharp\colon \pi_1(\mathcal{G},w) \to \pi_1(\mathcal{G}',f(w))$,
and we claim that it follows that the resulting  maps of graphs of groups are homotopic,
but again we leave the demonstration to the reader.
\begin{enumerate}
    \item Replace the edge path $f(e) \in \mathcal{G}'$ with a path which is homotopic
        to it rel endpoints
        and replace the homomorphisms $f_{e,e'_i}$ with new homomorphisms
        so that the compatibility conditions above are still satisfied.
        Note that since $f_v$ and $f_w$ are unchanged under this operation,
        the images of the new edge group homomorphisms must be abstractly isomorphic
        to the original $f_{e,e'_i}$.
        This restricts the possible homotopies we can perform,
        but does not prevent us from tightening $f(e)$ if it is not already tight.
    \item Suppose that $v \ne w$ is a vertex of $G$
        and that $g \in \mathcal{G}'_{f(v)}$.
        Replace $f_v$ with $\operatorname{ad}(g)f_v$
        and for each oriented edge $e$ of $G$ with initial vertex $v$,
        replace the edge path $f(e)$ with $gf(e)$.
    \item Suppose that $v \ne w$ is a vertex of $G$,
        that $e'$ is an edge of $G'$ with terminal vertex $f(v)$
        and initial vertex $v'$ such that $\iota_{e'}(\mathcal{G}'_{e'})$
        contains the image of $f_v$.
        Change the map $f$ of graphs by a homotopy supported in a neighborhood of $v$ 
        by pulling the image of $v$ across $\bar e'$
        so that the new map called $f'$ satisfies $f'(v) = v'$.
        The new map $f'_v$ is $\iota_{\bar e'}\iota_{e'}^{-1}f_v$.
        That is to say, the map $f'_v$ is accomplished by viewing $f_v(\mathcal{G}_v)$
        as a subgroup of $\mathcal{G}'_{e'}$ then mapping it to $\mathcal{G}'_{v'}$.
        For each oriented edge $e$ of $G$ with initial vertex $v$,
        we have the new edge path $f'(e) = e'f(e)$.
        Add the homomorphism $f'_{e,e'}\colon \mathcal{G}_e \to \mathcal{G}'_{e'}$
        defined as $f'_{e,e'}(h) = \iota_{e'}^{-1}f_v\iota_{\bar e}(h)$.
        Observe that the compatibility conditions are satisfied.
\end{enumerate}
If, ignoring the stipulations around the basepoint $w$,
the map $f\colon \mathcal{G} \to \mathcal{G}'$ can be transformed
to a map $f'\colon \mathcal{G} \to \mathcal{G}'$ by a finite number
of the above operations,
we will say that $f$ and $f'\colon \mathcal{G} \to \mathcal{G}'$ are \emph{homotopic.}

\paragraph{The Bass--Serre tree.}
Let $\mathcal{G}$ be a graph of groups and
let $p \in G$ be a basepoint.
For $v$ a vertex of $G$, write $[p,v]$ for the set of homotopy classes of paths in $\mathcal{G}$ 
from $p$ to $v$.
The group $\mathcal{G}_v$ acts on $[p,v]$ on the right:
an element $g$ sends the homotopy class $[\gamma]$ of the path $\gamma$ to 
the homotopy class of the composite path $[\gamma g]$.
Let $V$ denote the set of vertices of $G$.
The set
\[  \coprod_{v \in V}[p,v]/\mathcal{G}_v \]
forms the vertex set of a graph,
where two elements $[\gamma]\mathcal{G}_v$ and $[\gamma']\mathcal{G}_w$ are adjacent
if the path $\bar\gamma\gamma'$ is homotopic to a path of length one.
The fundamental group $\pi_1(\mathcal{G},p)$ acts naturally on this graph;
for example
the homotopy class of a loop $[\sigma]$ sends the vertex
$[\gamma]\mathcal{G}_v$ to the vertex $[\sigma\gamma]\mathcal{G}_v$.
The fundamental theorem of Bass--Serre theory asserts that this graph is a tree \cite[Theorem 1.17]{Bass},
and that the quotient graph of groups of the action  of $\pi_1(\mathcal{G},p)$ on this tree
may be identified with $\mathcal{G}$ \cite[Corollary 3.7]{Bass}, \cite[Chapter I, Theorem 13]{Trees}.
(Let us remark that non-vertex points of the Bass--Serre tree
may be identified with homotopy classes of paths in $\mathcal{G}$
that do not end at vertices.)

Indeed, suppose we began with a group $F$ acting on a tree $\Gamma$
with quotient graph of groups $\mathcal{G}$ defined relative to the
choice of spanning tree $S$ containing $p$, 
of fundamental domain $T$ and of group elements $g_e$ in $F$.
For $v$ a vertex of $G$, let $\gamma_v$ denote the unique tight path in $S$
as an ordinary graph from $p$ to $v$.
Then for $e$ an edge of $G$ and $g \in \mathcal{G}_v$ an element of a vertex group,
thought of as the stabilizer of $\tilde v$ in $F$,
the map
\[   [\gamma_{\tau(\bar e)}e\bar\gamma_{\tau(e)}] \mapsto g_e
\text{ and } [\gamma_{v}g\bar\gamma_v] \mapsto g\]
defines an isomorphism $\Phi\colon \pi_1(\mathcal{G},p) \to F$.
For $v$ a vertex of $\mathcal{G}$, let $\tilde v$ be the unique preimage of $v$ in $\tilde S$.
For $\tilde e$ an edge of $T \setminus \tilde S$, orient $\tilde e$ 
so that its initial vertex $\tilde w$ is in $\tilde S$ and its terminal vertex $\tilde w'$ is not.
The smallest subtree of the Bass--Serre tree containing the vertices
\[  [\gamma_v]\mathcal{G}_v \quad\text{and}\quad [\gamma_we]\mathcal{G}_{w'} \]
is a fundamental domain for the action of $\pi_1(\mathcal{G},p)$ on its Bass--Serre tree,
and the map
\[  [\gamma_v]\mathcal{G}_v \mapsto \tilde v \quad\text{and}\quad
[\gamma_we]\mathcal{G}_{w'} \mapsto \tilde w' \]
extends to a $\Phi$-twisted equivariant simplicial isomorphism between the Bass--Serre tree of $\mathcal{G}$
and $\Gamma$ taking this fundamental domain to $T$.

We will use $\Phi$ to identify $\pi_1(\mathcal{G},p)$ with $F$
and this $\Phi$-twisted equivariant simplicial isomorphism
to identify $\Gamma$ with the Bass--Serre tree of $\mathcal{G}$.
Write $\pi\colon\Gamma \to G$ for the natural projection.

\paragraph{Normal forms, projecting and lifting paths.}
Suppose $\gamma$ is a tight path in $\mathcal{G}$ 
from a point $x$ in the interior of an edge to a point $y$,
and that $\tilde x = [\sigma]$ is a lift of $x$ to $\Gamma$.
Then the unique tight path $\tilde\gamma$ from $\tilde x = [\sigma]$ to $[\sigma\gamma]$
(or $[\sigma\gamma]\mathcal{G}_y$ if $y$ is a vertex)
is a tight path in $\Gamma$ which \emph{lifts} $\gamma$ in the sense that
$\pi(\tilde\gamma)$ has the same underlying edge path in $G$ as $\gamma$.
If $x$ is a vertex,
then for each $g \in \mathcal{G}_x$,
there is a unique tight path $\tilde\gamma$ from $\tilde x = [\sigma]\mathcal{G}_x$
to $[\sigma g\gamma]$
(or $[\sigma g\gamma]\mathcal{G}_y$ if $y$ is a vertex)
which lifts $\gamma$.
If $\gamma$ is \emph{not} a tight path,
we may decompose it into a concatenation of tight paths
and successively lift those paths to obtain a lift of $\gamma$.
We would like to define a notion of \emph{projecting} paths from $\Gamma$ to $\mathcal{G}$
so that lifting and projecting are (nearly) inverse operations.

Recall \cite[1.12]{Bass} that given a choice,
for each oriented edge $e$ with terminal vertex $v$,
of a set $S_e$ of left coset representatives for $\mathcal{G}_v/\iota_e(\mathcal{G}_e)$
containing $1 \in \mathcal{G}_{v}$,
there is a \emph{normal form} for edge paths in $\mathcal{G}$,
where an edge path
\[  \gamma = e'_1g_1e_2\ldots g_{k-1}e_kg_ke'_{k+1} \]
is in normal form if it is tight and each $g_i$ for $1 \le i \le k-1$ belongs to $S_{\bar e_{i+1}}$.
A path can be inductively put in normal form by performing homotopies from ``left to right.''

Suppose $\tilde x$ is a point $[\sigma]$ 
(or $[\sigma]\mathcal{G}_x$ if $\tilde x$ is a vertex)
of $\Gamma$ and that $\tilde y$ is a point $[\eta]$
(or $[\eta]\mathcal{G}_y$ if $\tilde y$ is a vertex)
of $\Gamma$,
and let $\tilde\gamma$ be the unique tight path from $\tilde x$ to $\tilde y$.
If $\tilde y$ is a vertex, choose $\eta$ in its $\mathcal{G}_y$-orbit
so that its final vertex group element is $1$.
Let $\gamma = \pi(\tilde\gamma)$ be the unique path in normal form homotopic to $\bar\sigma\eta$,
thus  $\gamma$ has the form
\[  \gamma = e'_1g_1e_2\ldots g_{k-1}e_k1 \]
where each $g_i$ for $1 \le i \le k-1$ belongs to $S_{\bar e_{i+1}}$.
One checks that if $\tilde x$ is not a vertex,
then the lift of $\gamma$ defined above is $\tilde\gamma$.
If $\tilde x = [\sigma]\mathcal{G}_x$ is a vertex and choose $\sigma$ in its $\mathcal{G}_x$-orbit
so that its final vertex group element is $1$,
then the lift of $\gamma$ defined above corresponding to the choice of $1 \in \mathcal{G}_x$
is $\tilde\gamma$.

Conversely, if we first lift the normal-form path $\gamma$ to $\tilde\gamma$ and then project
to $\pi(\tilde\gamma)$ as in the previous paragraph, then $\pi(\tilde\gamma)$ and $\gamma$
agree except for the final vertex group element if $y$ is a vertex.

\begin{prop}[Lifting maps, cf.~Proposition 2.4 of \cite{Bass}]
    \label{liftingtotheuniversalcover}
    Suppose that $f\colon (\mathcal{G},v) \to (\mathcal{G}',v')$
    is a map of graphs of groups
    and that $(\Gamma,\tilde v)$ and $(\Gamma',\tilde v')$ are the Bass--Serre trees
    of $\mathcal{G}$ and $\mathcal{G}'$ respectively.
    There exists an $f_\sharp$-twisted equivariant map 
    $\tilde f\colon (\Gamma,\tilde v) \to (\Gamma',\tilde v')$
    such that the following diagram commutes
    as maps of underlying graphs
    \[  \begin{tikzcd}
        (\Gamma,\tilde v) \ar[r, dashed, "\tilde f"] \ar[d, "\pi"] & (\Gamma',\tilde v') \ar[d, "\pi'"] \\
        (\mathcal{G},v) \ar[r, "f"] & (\mathcal{G}',v'),
    \end{tikzcd}\]
    where $\pi$ and $\pi'$ are the natural projections.
    Furthermore, the normal form of the path $f\pi(\tilde\gamma)$ agrees with the
    path $\pi'f(\tilde\gamma)$ except for the final vertex group element if $\tilde\gamma$
    ends at a vertex.
    If the maps $f$ and $f'$ are homotopic fixing $v$,
    then the lifted maps $\tilde f$ and $\tilde f'$ are equivariantly homotopic.
\end{prop}

It follows from the final claim that if $f$ is a homotopy equivalence, then $\tilde f$ is too.

\begin{proof}
    A point $\tilde p \in \Gamma$ corresponds to a homotopy class $[\gamma]$ 
    of paths from $v$ to $p = \pi(\tilde p)$
    (or to $[\gamma]\mathcal{G}_p$ if $p$ is a vertex).
    Define $\tilde f(\tilde p) = [f(\gamma)]$
    (or $[f(\gamma)]\mathcal{G}'_{f(p)}$ if $f(p)$ is a vertex).
    A simple calculation shows that $\tilde f$ is $f_\sharp$-twisted  equivariant
    and that the claimed properties hold.
\end{proof}

\begin{prop}[Projecting maps, cf.~4.1--4.5 of \cite{Bass}]
    \label{takingthequotient}
    Suppose $F$ and $F'$ are groups acting on trees $\Gamma$ and $\Gamma'$ respectively
    with quotient graphs of groups $\mathcal{G}$ and $\mathcal{G}'$.
    Let $\Phi\colon F \to F'$ be a homomorphism
    and let $\tilde f\colon (\Gamma,\tilde v) \to (\Gamma',\tilde v')$
    be a $\Phi$-twisted equivariant map of trees.
    There is a map $f\colon (\mathcal{G},v) \to (\mathcal{G}',v')$
    of graphs of groups
    such that the following diagram commutes as maps of underlying graphs
    \[  \begin{tikzcd}
        (\Gamma,\tilde v) \ar[r, "\tilde f"] \ar[d, "\pi"] & (\Gamma',\tilde v') \ar[d, "\pi'"] \\
        (\mathcal{G},v) \ar[r, dashed, "f"] & (\mathcal{G}',v'),
    \end{tikzcd}  \]
    where $\pi$ and $\pi'$ are the natural projections.
    Furthermore, the normal form of the path $f\pi(\tilde\gamma)$ agrees with the path $\pi'\tilde f(\tilde\gamma)$
    except for the final vertex group element if $\tilde\gamma$ ends at a vertex.
    If two maps $\tilde f$ and $\tilde f'$ are $\Phi$-twisted equivariantly homotopic
    then $f$ and $f'$ are homotopic.
\end{prop}

\begin{proof}
    By $\Phi$-twisted equivariance, the map $\pi'\tilde f$ yields a well-defined map
    on $\pi_1(\mathcal{G},v)$-orbits;
    this is the map $f\colon G \to G'$ as a map of graphs.

    Let $T \subset \Gamma$ and $T'  \subset \Gamma'$ be fundamental domains
    containing $\tilde v$ and $\tilde v'$ respectively
    and let $\mathcal{G}$ and $\mathcal{G}'$ be the quotient graphs of groups
    associated to $T$ and $T'$ respectively.
    Each edge $e \in G$ and $e' \in G'$
    has a single preimage $\tilde e \in T$ and $\tilde e' \in T'$ respectively.
    The groups $\mathcal{G}_{e}$ and $\mathcal{G}'_{e'}$  are  the stabilizers
    of $\tilde e$ and $\tilde e'$ respectively.
    Each vertex $v \in G$ and $v' \in G'$ has a preferred preimage
    $\tilde v \in T$ and $\tilde v' \in T'$ respectively.
    The groups $\mathcal{G}_v$ and $\mathcal{G}'_{v'}$ are the stabilizers
    of $\tilde v$ and $\tilde  v'$ respectively.
    For each oriented edge $e$ of $G$ and $e'$ of $G'$
    there are elements $g_e \in F$ and $g'_{e'} \in F'$
    such that the monomorphisms $\iota_e$ and $\iota_{e'}$ are the restrictions of the maps
    \[  h \mapsto g_ehg_e^{-1} \quad\text{and}\quad h' \mapsto g'_{e'}h'g'^{-1}_{e'} \]
    to $\mathcal{G}_e$ and $\mathcal{G}'_{e'}$ respectively.

    Let $v$ be a vertex of $G$ and write $w = f(v)$.
    To define $f_v\colon \mathcal{G}_v \to \mathcal{G}'_w$,
    note that the stabilizers of $\tilde f(\tilde v)$ and $\tilde w$ are conjugate
    in $F'$ by some element $t_w$ such that $t_w.\tilde f(\tilde v) = \tilde w$.
    The restriction of
    \[  h \mapsto t_w\Phi(h)t_w^{-1} \]
    to $\mathcal{G}_v$ defines a homomorphism $f_v\colon \mathcal{G}_v \to \mathcal{G}'_{w}$.

    Let $e$ be an oriented edge of $G$ with initial vertex  $v$ and terminal vertex $w$
    and suppose first that $f(e)$ is a vertex.
    Then for $h \in \mathcal{G}_e$, we have
    \[  f_v\iota_{\bar e}(h) = t_{f(v)}\Phi(g_{\bar e}hg_{\bar e}^{-1})t_{f(v)}^{-1}  \]
    and
    \[  f_w\iota_{e}(h) = t_{f(w)}\Phi(g_ehg_{e}^{-1})t_{f(w)}^{-1}. \]
    Therefore the compatibility conditions force us to define $f(e) = g_0$,
    where
    \[
        g_0 = t_{f(v)}\Phi(g_{\bar e}g_{e}^{-1})t_{f(w)}^{-1}.
    \]
    One checks that $g_0$ belongs to $\mathcal{G}_{f(v)} = \mathcal{G}_{f(w)}$.
    
    Now suppose that $\tilde f(\tilde e) = \tilde E'_1\ldots \tilde E'_k$.
    $\Phi$-twisted equivariance implies that if $h \in \mathcal{G}_e$,
    then $\Phi(h)$ stabilizes $\tilde E'_i$ for $1 \le i \le k$.
    Suppose $\pi'(\tilde E'_i) = e'_i$ in $G'$.
    These stabilizers are conjugate to $\mathcal{G}'_{e'_i}$
    by elements $t_{e'_i} \in F'$ such  that $t_{e'_i}.\tilde E'_i = \tilde e'_i$.
    The restriction of
    \[  h \mapsto t_{e'_i}\Phi(h)t_{e'_i}^{-1} \]
    to $\mathcal{G}_e$ defines a homomorphism  $f_{e,e'_i}\colon \mathcal{G}_e \to  \mathcal{G}'_{e'_i}$.
    Given $h \in \mathcal{G}_e$, we have
    \[  f_v\iota_{\bar e}(h) = t_{f(v)}\Phi(g_{\bar e}hg_{\bar e}^{-1})t_{f(v)}^{-1}. \]
    On the other hand, we have
    \[  \iota_{\bar e'_1}f_{e,e'_1}(h) = g'_{\bar e'_1}t_{e'_1}\Phi(h)t_{e'_1}^{-1}g_{\bar e'_1}^{-1}. \]
    The compatibility conditions force us to define
    \[  g_0 = t_{f(v)}\Phi(g_{\bar e})t_{e'_1}^{-1}g'^{-1}_{\bar e'_1}. \]
    One checks that we have $g_0 \in \mathcal{G}_{f(v)}$.
    The derivations of the other vertex  group elements are similar.
    The additional claims are straightforward to check; we leave them to the reader.
\end{proof}

Given a vertex $v \in G$, let $\st(v)$ denote the set of oriented edges $e$ of $G$
with initial vertex $v$.
Recall that for each lift $\tilde v \in \pi^{-1}(v)$
there is a $\mathcal{G}_v$-equivariant bijection
\[  \st(\tilde v) = \coprod_{e \in \st(v)} \mathcal{G}_v/\iota_e(\mathcal{G}_e) \times \{e\}. \]

Suppose $f\colon \mathcal{G} \to \mathcal{G}$ is a homotopy equivalence.
It defines an outer automorphism of $\pi_1(\mathcal{G})$.
Choosing a basepoint $p$ in $G$ and an edge path $\sigma$ in $\mathcal{G}$ from $p$ to $q = f(p)$
defines an automorphism $f_\sharp \colon \pi_1(\mathcal{G},p) \to \pi_1(\mathcal{G},p)$
defined as
\[  f_\sharp([\gamma]) = [\sigma f(\gamma)\bar\sigma]. \]
Let $(\Gamma,\tilde p)$ be the Bass--Serre tree covering $(\mathcal{G},p)$, and
let $\tilde q$ be the endpoint of the lift of $\sigma$ to $\Gamma$ beginning at $\tilde p$.
By \Cref{liftingtotheuniversalcover}, there is an $f_\sharp$-twisted equivariant map
$\tilde f\colon (\Gamma,\tilde p) \to (\Gamma,\tilde q)$
taking the point $\tilde x = [\gamma]$ 
(or $\tilde x = [\gamma]\mathcal{G}_x$ if $x$ is a vertex)
to the point $\tilde f(\tilde x) = [\sigma f(\gamma)]$ 
(or $\tilde f(\tilde  x) = [\sigma f(\tilde \gamma)]\mathcal{G}_{f(x)}$ if $f(x)$ is a vertex).

Conversely, we saw earlier in the definition of a map of graphs of groups
and \Cref{takingthequotient}
that any $\Phi$-twisted equivariant map $\tilde f\colon \Gamma \to \Gamma'$
is equivariantly homotopic to a map which projects to a map of graphs of groups.
It is not quite true that if $\Phi$ is an automorphism, then $\tilde f$ projects to a homotopy equivalence:
for that one needs the existence of a $\Phi$-twisted equivariant homotopy inverse for $\tilde f$.

\paragraph{Topological representatives.}
Recall that we defined
a \emph{topological representative}
to be a homotopy equivalence $f\colon \mathcal{G} \to \mathcal{G}$
which is a morphism and which sends edges to nontrivial tight edge paths.
In the following sections, we prove \Cref{relativetraintrack} 
by performing a number of operations on topological representatives.

Given a graph of groups $\mathcal{G}$,
the collection of outer automorphisms of $F = \pi_1(\mathcal{G})$
that admit a topological representative $f\colon \mathcal{G} \to \mathcal{G}$
forms a subgroup of $\out(F)$.
It is precisely the subgroup of $\out(F)$ leaving invariant the deformation space $\mathscr{D}$
to which the Bass--Serre tree of $\mathcal{G}$ belongs.
We suggest the name \emph{modular group} or \emph{mapping class group} of $\mathcal{G}$
for this group, and in this paper we will write it as $\Mod(\mathscr{D})$ or $\Mod(\mathcal{G})$.
An outer automorphism $\varphi \in \out(F)$ belongs to the ``modular group'' of $\mathcal{G}$
if for any conjugacy class $[g]$ in $F$, the conjugacy class $\varphi([g])$
is elliptic in the Bass--Serre tree for $\mathcal{G}$ if and only if $[g]$ is elliptic.
In some cases this subgroup is all of $\out(F)$.
One case where this happens is when $F$ is virtually free
and vertex groups of $\mathcal{G}$ are finite.
Another is when $F$ is a \emph{generalized Baumslag--Solitar group}
and vertex and edge groups of $\mathcal{G}$ are infinite cyclic.

Another case where this ``modular group'' is all of $\out(F)$
is when $F = A_1*\dotsb*A_n*F_k$ is the Grushko decomposition of a finitely generated group.
That is, the $A_i$ are freely indecomposable and not infinite cyclic
and $F_k$ is a free group of rank $k$.
The graph of groups $\mathcal{G}$ can 
be any graph of groups with trivial edge groups,
vertex groups the $A_i$ and ordinary fundamental group free of rank $k$.
For example, $\mathcal{G}$ may be the \emph{thistle with $n$ prickles and $k$ petals.}
This is a graph of groups with one vertex $\star$ with trivial vertex group,
$n$ vertices with vertex group each of the $A_i$, and $n+k$ edges.
The first $n$ edges connect vertices with nontrivial vertex group to $\star$,
and the remaining $k$ edges form loops based at $\star$.

\begin{ex}
    \label{irreducibleexample}
    Consider
    \[  F = C_2*C_2*C_2*C_2 = \langle a,b,c,d \mid a^2 = b^2 = c^2 = d^2 = 1\rangle \]
    the free product of four copies of the cyclic group of order two.
    Let $\Phi\colon F \to F$ be the automorphism
    \[ \Phi \begin{dcases}
        a \mapsto b \\
        b \mapsto c \\
        c \mapsto d \\
        d \mapsto cbdadbc.
    \end{dcases} \]
    (Notice that, e.g.~$c^{-1}= c$.)
    A topological representative $f\colon \mathcal{G} \to \mathcal{G}$ of $\Phi$
    on the thistle with four prickles is depicted in \Cref{traintrackfig1}.
    The maps on vertex groups are the unique isomorphisms.
    \begin{figure}[ht]
        \[ 
        \begin{tikzpicture}[auto, node distance = 4cm, baseline]
            \node[pt, star, scale = 1.5] (v) at (0,0) {};
            \node[pt, "$\langle a \rangle$" left] (a) [above left of = v] {}
                edge[->-, "$e_1$"] (v);
            \node[pt, "$\langle b\rangle$" right] (b) [above right of = v] {}
                edge[->-, "$e_2$"] (v);
            \node[pt, "$\langle c\rangle$" right] (c) [below right of = v] {}
                edge[->-, "$e_3$"] (v);
            \node[pt, "$\langle d \rangle$" left] (d) [below left of = v] {}
                edge[->-, "$e_4$"] (v);
        \end{tikzpicture}
        \qquad
        f \begin{dcases}
            e_1 \mapsto e_2 \\
            e_2 \mapsto e_3 \\
            e_3 \mapsto e_4 \\
            e_4 \mapsto e_1\bar e_4de_4\bar e_2be_2\bar e_3ce_3
        \end{dcases}
        \]
        \caption{The topological representative $f\colon \mathcal{G} \to \mathcal{G}$.}
        \label{traintrackfig1}
    \end{figure}
\end{ex}

\section{Train Track Maps}
\label{traintracksection}
The purpose of this section is to prove the irreducible case of \Cref{relativetraintrack}.
The strategy is a straightforward adaptation of the arguments of \cite[Section 1]{BestvinaHandel} to graphs of groups.
At the end of the section we prove a proposition characterizing irreducibility
for outer automorphisms of free products.

\paragraph{}
Fix once and for all a graph of groups $\mathbb{G}$.
A \emph{marked graph of groups} is a graph of groups $\mathcal{G}$ together
with a homotopy equivalence $\sigma\colon \mathbb{G} \to \mathcal{G}$.
In the language of \cite{GuirardelLevittDeformation},
the marking keeps track not only of the fundamental group of $\mathcal{G}$,
but constrains the \emph{deformation space} to which it belongs.
We will assume that $\mathbb{G}$ is \emph{reduced} in the sense of \cite{Forester},
i.e.~there is no homotopy equivalence that collapses an edge of $\mathbb{G}$.

Given a topological representative $f\colon \mathcal{G} \to \mathcal{G}$
and an ordering $e_1,\dotsc,e_m$ of the edges of $G$,
there is an associated $m\times m$ \emph{transition matrix} $M$
with $ij$th entry counting the number of times 
the $f$-image of the $j$th edge crosses the $i$th edge in either direction.
The map $f$ is \emph{irreducible} if the matrix $M$ is irreducible.
Recall that a matrix is \emph{irreducible} if for $1 \le i,j \le m$,
there is an integer $\ell$ such that the $ij$th entry of $M^\ell$ is positive.
Associated to every irreducible matrix is its \emph{Perron--Frobenius eigenvalue} $\lambda \ge 1$.
An irreducible matrix with Perron--Frobenius eigenvalue $\lambda = 1$ is a transitive permutation matrix.
The transition matrix of \Cref{irreducibleexample} is
\[ \begin{pmatrix}
	0 & 0 & 0 & 1 \\
	1 & 0 & 0 & 2 \\
	0 & 1 & 0 & 2 \\
	0 & 0 & 1 & 2 
\end{pmatrix}, \]
which is irreducible and
for which the Perron--Frobenius eigenvalue is 
the largest real root of the polynomial $x^4 - 2x^3 - 2x^2 - 2x - 1$
and satisfies $\lambda \approx 2.948$.

Call a vertex $v$ of $\mathcal{G}$ \emph{inessential} if for some oriented edge $e$ with terminal vertex $v$
the homomorphism $\iota_e\colon\mathcal{G}_e \to \mathcal{G}_v$ is surjective.

A subgraph $G_0$ of $G$ is \emph{invariant} with respect 
to a map $f\colon \mathcal{G} \to \mathcal{G}$
if $f(G_0) \subset G_0$. It is a \emph{forest}
if each component $C$ of $G_0$ is a tree and in the induced graph of groups structure
we have that $\pi_1(\mathcal{G}|_C)$ acts with global fixed point on its Bass--Serre tree.
In $C$, this means there is a choice of vertex $v$ in $C$ and an orientation
of each edge $e$ of $C$ toward this vertex such that 
each homomorphism $\iota_{\bar e}$ away from $v$ is surjective.
A forest is \emph{nontrivial} if it contains at least one edge.
An outer automorphism $\varphi \in \out(\pi_1(\mathcal{G}))$ is \emph{irreducible}
if it admits a topological representative $f\colon \mathcal{G} \to \mathcal{G}$
(i.e.~we have $\varphi \in \Mod(\mathcal{G})$)
and if whenever $\mathcal{G}$ has no inessential valence-one vertices and no nontrivial invariant forests,
then the topological representative $f$ is irreducible.

A homotopy equivalence $f\colon \mathcal{G} \to \mathcal{G}$ (taking vertices to vertices) is \emph{tight} 
if for each edge $e$, either $f(e)$ is a tight edge path, or $f(e)$ is a vertex.
A homotopy equivalence may be \emph{tightened} to a tight homotopy equivalence
by a homotopy relative to the vertices of $G$.
In the language of the previous section, in other words,
the homotopy only involves the first operation in the definition of homotopy of maps of graphs of groups.

Suppose $f\colon \mathcal{G} \to \mathcal{G}$ is a tight homotopy equivalence.
A forest $G_0 \subset \mathcal{G}$ is \emph{pretrivial} if 
each edge in the forest is eventually mapped to a point.
\emph{Maximal} (with respect to inclusion) pretrivial forests are in particular invariant.

\begin{lem}[\cite{BestvinaHandel} p.~7] 
	If $f\colon \mathcal{G} \to \mathcal{G}$
	is a tight homotopy equivalence,
	collapsing a maximal pretrivial forest in $\mathcal{G}$
	produces a topological representative $f'\colon \mathcal{G}' \to \mathcal{G}'$.
	If instead $f\colon \mathcal{G} \to \mathcal{G}$
	is a topological representative of an irreducible outer automorphism
	and $G$ has no inessential valence-one vertices,
	collapsing a maximal invariant forest yields an irreducible topological representative
	$f'\colon\mathcal{G}' \to \mathcal{G}'$.
\end{lem}

\begin{proof}
	We describe how to collapse invariant forests.

	If $f\colon\mathcal{G} \to \mathcal{G}$ is a tight homotopy equivalence
	and $G_0 \subset G$ is an invariant forest,
	define $\mathcal{G}_1 = \mathcal{G}/G_0$ to be the quotient graph of groups 
	obtained by collapsing each component $C$ of $G_0$ to a vertex.
	The vertex group of the vertex determined by $C$ is the fundamental group $\pi_1(\mathcal{G}|_C,p_C)$
    with respect to some basepoint $p_C \in C$.
	Since $G_0$ is a forest, this fundamental group is equal to some vertex group in $C$.
    Choose $p_C$ equal to that vertex; choose arbitrarily if there are multiple choices.
    Given a vertex $v$ of $C$, let $\gamma_v$ be unique tight path without vertex group elements
    from $v$ to $p_C$.
    If $v$ does not belong to any component of $G_0$,
    let $\gamma_v$ be the trivial path (without vertex group elements).
	Let $\pi\colon \mathcal{G} \to \mathcal{G}_1$ be the quotient map.
    It is a collapse map of graphs of groups.

    Each edge $e$ of $G_1$ has a unique preimage in $G$; abusing notation, call it $e$ as well.
    Define $f_1(e) = \pi f(\bar\gamma_{\tau(\bar e)}e\gamma_{\tau(e)})$.
    If $v$ in $G_1$ is a vertex, then $v$ either corresponds to a unique vertex of $G$,
    call it $v$ as well,
    or to a component $C$ of $G_0$.
    In the former case define $(f_1)_v = f_v$.
    Suppose in the latter case that $f$ maps the component $C$ to $C'$.
    Let $\gamma_{p_C'}$ be the unique tight path in $C'$ without vertex group elements
    from $f(p_C)$ to $p_C'$;
    this determines a map 
    $f_\sharp\colon \pi_1(\mathcal{G}|_C,p_C) \to \pi_1(\mathcal{G}|_{C'},p_{C'})$;
    this is the map $(f_1)_v$ in this case.

	If $e \subset G$ is an edge not in $G_0$, 
	then the edge path for $f_1(e)$ is obtained from $f(e)$ by deleting all occurrences of edges in $G_0$
    and possibly adding vertex group elements at the ends.
	Since $f$ was tight, if $e\sigma\bar e$ is a subpath of the $f$-image of some edge $e'$ not in $G_0$,
	where $\sigma$ is a nontrivial path in $G_0$,
	then $\sigma$ must be homotopic to a path 
    of the form $\sigma'g\bar\sigma'$ for some path $\sigma'$ in $G_0$
    without vertex group elements
	and $g$ an element of some vertex group.
	In $f_1(e')$, the path $e\sigma\bar e$ is replaced by $eg\bar e$.
	This implies that $f_1\colon\mathcal{G}_1 \to \mathcal{G}_1$ is tight.
    Moreover, if $G_0$ was a maximal pretrivial forest,
    then $f_1$ is a topological representative.
    
    If instead $f$ was a topological representative of an irreducible outer automorphism, then
	the transition matrix for $f_1\colon\mathcal{G}_1\to \mathcal{G}_1$ 
    is obtained from the transition matrix
	for $f\colon\mathcal{G} \to \mathcal{G}$ 
    by deleting the rows and columns associated to the edges of $G_0$.
    If in this latter case $G$ had no inessential valence-one vertices,
    then neither does $G_1$,
    and since $G_0$ was a maximal invariant forest,
    $G_1$ has no invariant forests; therefore $f\colon \mathcal{G}_1 \to \mathcal{G}_1$
    must be irreducible.
\end{proof}

Recall we write $\st(v)$ for the set of oriented edges $e$ with initial vertex $v$.
A \emph{direction} at $v$ is an element of the set
\[ \coprod_{e \in \st(v)}\mathcal{G}_v/\iota_e(\mathcal{G}_e)\times \{ e\}. \]
A \emph{turn} at $v$ is a pair of directions at $v$.
If $f\colon \mathcal{G} \to \mathcal{G}$ is a topological representative,
$f$ determines a map $Df$ on directions
sending a direction based at $v$ to a direction based at $f(v)$ via the rule
\[ ([g],e) \mapsto ([f_{v}(g)g_0],e_1), \]
where the edge path $f(e)$ begins with $g_0e_1$.
The compatibility condition ensures that this map is well-defined; we have
\[ f_v(g\iota_e(h))g_0 = f_v(g)f_v(\iota_e(h))g_0 = f_v(g)g_0\iota_{\bar e_1}(f_{e,e_1}(h)). \]
The map $Df$ induces a map on turns, which we also denote by $Df$.
In \Cref{irreducibleexample}, the vertex $\star$ is mapped to itself by $f$;
the restriction of $Df$ to $\star$ is determined by the dynamical system
$\bar e_1 \mapsto \bar e_2 \mapsto \bar e_3 \leftrightarrow \bar e_4$.

A turn is \emph{degenerate} if it consists of a pair of identical elements and is \emph{nondegenerate} otherwise.
A turn is \emph{illegal} with respect to a topological representative $f\colon \mathcal{G} \to \mathcal{G}$
if its image under some iterate of $Df$ is degenerate and is \emph{legal} otherwise.
In \Cref{irreducibleexample}, a turn $\{\bar e_i,\bar e_j\}$ 
based at $\star$ is illegal if $i$ and $j$ are equal mod $2$,
and is legal otherwise.

Consider the edge path
\[ \gamma = g_1e_1g_2e_2\ldots e_kg_{k+1}. \]
We say $\gamma$ \emph{takes} the turns $\{([1],\bar e_i),([g_{i+1}],e_{i+1})\}$.
The path $\gamma$ is \emph{legal} if it takes only legal turns.

A topological representative $f\colon \mathcal{G} \to \mathcal{G}$ is a \emph{train track map}
if $f(e)$ is a legal path for each edge $e$ of $\Gamma$.
Equivalently, $f$ is a train track map if for each $k \ge 1$
and each edge $e$ of $\Gamma$, we have that
$f^k(e)$ is a tight edge path.
In \Cref{irreducibleexample}, $f$ is not a train track map
because the image of $e_4$ takes the illegal turn $\{\bar e_4,\bar e_2\}$.

\paragraph{\Cref{irreducibleexample} Continued.}
Let us fold $f$ at the illegal turn $\{\bar e_2,\bar e_4\}$.
To do this, first subdivide $e_4$ at the preimage of the vertex with vertex group $\langle c\rangle$
so $e_4$ becomes the edge path $e'_4e''_4$ and identify $e''_4$ with $e_2$.
The action of the resulting map $f'\colon \mathcal{G}_1 \to \mathcal{G}_1$
is obtained from $f$ by replacing instances of $e_4$ with $e'_4e_2$.
Thus we have
\[ f'(e_4) = e_1\bar e_2\bar e_4'de'_4e_2\bar e_2be_2\bar e_3c. \]
Tighten $f'$ by a homotopy with support on $e'_4$ to remove $e_2\bar e_2$,
yielding an irreducible topological representative $f_1\colon\mathcal{G}_1 \to \mathcal{G}_1$.
See \Cref{traintrackfig2}.

\begin{figure}[ht]
	\begin{center}
		\[
			\begin{tikzpicture}[auto, node distance = 4cm, baseline]
				\node[pt, star, scale = 1.5] (v) at (0,0) {};
				\node[pt, "$\langle a\rangle$" left] (a) [above of = v] {}
					edge[->-, "$e_1$"] (v);
				\node[pt, "$\langle b\rangle$" above] (b) [right of = v] {}
					edge[->-, "$e_2$"] (v);
				\node[pt, "$\langle c\rangle$" left] (c) [below of = v] {}
					edge[->-, "$e_3$"] (v);
				\node[pt, "$\langle d\rangle$" above] (d) [right of = b] {}
					edge[->-, "$e'_4$"] (b);
			\end{tikzpicture}
			\qquad
			f_1\begin{dcases}
				e_1 \mapsto e_2 \\
				e_2 \mapsto e_3\\
				e_3 \mapsto e'_4e_2 \\
				e'_4 \mapsto e_1\bar e_2\bar e'_4 d e'_4 b e_2\bar e_3 c
			\end{dcases}
		\]
		\caption{The topological representative $f_1\colon\mathcal{G}_1 \to \mathcal{G}_1$.}
		\label{traintrackfig2}
	\end{center}
\end{figure}

The Perron--Frobenius eigenvalue $\lambda_1$ for $f_1\colon \mathcal{G}_1 \to \mathcal{G}_1$
is the largest real root of the polynomial $x^4 - 2x^3 - 2x^2 + x - 1$ and
satisfies $\lambda_1 \approx 2.663$; thus $\lambda_1 < \lambda$.
However, $f_1$ is still not a train track map: $Df_1$ sends the turn
$\{(1,\bar e'_4),(b,e_2)\}$, which is crossed by $f_1(e'_4)$,
to $\{(c,e_3),(c,e_3)\}$; thus this turn is illegal.
We cannot quite fold $e_2$ and the end of $\bar e'_4$ 
because the $f_1$-image of the latter ends with $\bar e_3c$.
Lifting to the Bass--Serre tree $\tilde f_1 \colon \Gamma_1 \to \Gamma_1$, 
it is not the edge $\tilde e'_4$ which is folded with $\tilde e_2$ but $b.\tilde e'_4$.
We may remedy the situation by changing the fundamental domain in $\Gamma_1$,
or equivalently by changing the marking on $\mathcal{G}_1$ by twisting the edge $e'_4$ by $b^{-1} = b$.
This replaces $\langle d\rangle$ with $\langle bdb\rangle$, replaces $f_1(e_3)$ with $e'_4be_2$
and replaces $f_1(e'_4)$ with $e_1\bar e_2b\bar e'_4de'_4e_2\bar e_3$.
Then we fold $e'_4$ and $\bar e_2$.
The resulting graph of groups $\mathcal{G}_2$ 
is abstractly isomorphic to our original graph of groups $\mathcal{G}$,
but the marking differs.
The action of the resulting map $f''\colon \mathcal{G}_2 \to \mathcal{G}_2$
on edges is obtained by replacing instances of $e'_4$ with $e''_4\bar e_2$.
Thus we have
\[ f''(e''_4) = e_1\bar e_2 be_2e''_4bdbe''_4\bar e_2e_2, \]
and we may tighten to produce an irreducible topological representative
$f_2\colon \mathcal{G}_2 \to \mathcal{G}_2$. See \Cref{traintrackfig3}.
 \begin{figure}[ht]
	 \begin{center}
		 \[
			 \begin{tikzpicture}[auto, node distance = 4cm, baseline]
				 \node[pt, star, scale = 1.5] (v) at (0,0) {};
				 \node[pt, "$\langle a \rangle$" left] (a) [above left of = v] {}
					 edge[->-, "$e_1$"] (v);
				 \node[pt, "$\langle b\rangle$" right] (b) [above right of = v] {}
					 edge[->-, "$e_2$"] (v);
				 \node[pt, "$\langle c \rangle$" left] (c) [below left of = v] {}
					 edge[->-, "$e_3$"] (v);
				 \node[pt, "$\langle bdb\rangle$" right] (d) [below right of = v] {}
					 edge[->-, "$e''_4$"] (v);
			 \end{tikzpicture}
			 \qquad
			 f_2\begin{dcases}
				 e_1 \mapsto e_2 \\
				 e_2 \mapsto e_3 \\
				 e_3 \mapsto e''_4\bar e_2be_2 \\
				 e_4 \mapsto e_1\bar e_2be_2\bar e''_4bdbe''_4
			 \end{dcases}
		 \]
		 \caption{The topological representative $f_2\colon \mathcal{G}_2 \to \mathcal{G}_2$.}
		 \label{traintrackfig3}
	 \end{center}
 \end{figure}
The Perron--Frobenius eigenvalue $\lambda_2$ is the largest real root of $x^4 - 2x^3 - 2x^2 + 2x - 1$
and satisfies $\lambda_2 \approx 2.539$; thus $\lambda_2 < \lambda_1$.
The restriction of $Df_2$ to turns incident to $\star$ is determined by the dynamical system
$\bar e_1 \mapsto \bar e_2 \leftrightarrow \bar e_3$, $\bar e_4 \mapsto \bar e_4$.
The only illegal turn in $\mathcal{G}_2$ is $\{\bar e_1,\bar e_3\}$,
which is not crossed by the $f_2$-image of any edge, so $f_2\colon \mathcal{G}_2 \to \mathcal{G}_2$
is a train track map.

A subgroup $H$ of $\pi_1(\mathbb{G},p)$ is a \emph{generalized edge group}
if there exist edges $e$ and $e'$ of $G$ with terminal vertices $v$ and $v'$
and paths $\sigma$ from $p$ to $v$ and $\sigma'$ from  $p$ to $v'$
such that each element of $H$ may be represented by a loop of the form $\sigma\iota_e(h)\bar\sigma$
for $h \in \mathbb{G}_e$
and if $H$ contains all elements of the form $\sigma'\iota_{e'}(h')\bar\sigma'$
for $h' \in \mathbb{G}_{e'}$.
In the language of the action of $\pi_1(\mathbb{G},p)$ on the Bass--Serre tree $\Gamma$,
the first condition says that there is some edge $\tilde e$ lifting $e$
such that $H$ is contained in the stabilizer of $\tilde e$,
while the second says that there is some edge $\tilde e'$ lifting $e'$
such that $H$ contains the stabilizer of $\tilde e'$.
All reduced graphs of groups homotopy equivalent to $\mathbb{G}$
have the same generalized edge groups by \cite[Proposition 4.6]{GuirardelLevittDeformation},
and by its proof,
we have that any element $\varphi \in \Mod(\mathbb{G})$
permutes the conjugacy classes of generalized edge groups in $\pi_1(\mathbb{G})$.

An edge $e$ of a graph of groups $\mathcal{G}$ is \emph{surviving}
if there is a collapse map $\mathcal{G} \to \mathcal{G}'$
with $\mathcal{G}'$ reduced such that $e$ is not collapsed.
Suppose the Bass--Serre tree of $\mathcal{G}$ belongs to a deformation space $\mathscr{D}$.
Guirardel--Levitt consider a space $P\mathscr{G}$ of trees
all of whose edges are surviving
and prove \cite[Theorem 7.6]{GuirardelLevittDeformation}
that if $\mathscr{D}$ is what is called \emph{non-ascending,}
then $P\mathscr{G}$ is a finite-dimensional deformation retract 
of the projectivized deformation space $P\mathscr{D}$.
Clay \cite[Discussion after Lemma 1.11]{Clay} proves that
if we work in the \emph{weak topology} and $\mathcal{G}$ is \emph{irreducible}
then there is still a deformation retraction from $P\mathscr{D}$
to the \emph{spine} of $P\mathscr{G}$.
$\Mod(\mathscr{D}) = \Mod(\mathcal{G})$ acts on $P\mathscr{G}$ and its spine.
There are finitely many isomorphism types of graphs of groups $\mathcal{G}'$
homotopy equivalent to $\mathcal{G}$
all of whose edges are surviving
precisely when this action has finitely many orbits of cells.
For \emph{generalized Baumslag--Solitar groups}
where all vertex and edge groups are infinite cyclic,
Forester \cite[Theorem 8.2]{ForesterGBS} 
proved that this happens when there is no \emph{nontrivial integer modulus.}

The main result of this section is the following theorem.

\begin{thm}
	\label{traintrackthm}
    Suppose $\varphi \in \Mod(\mathbb{G})$ is irreducible
    and that one of the following conditions holds.
    \begin{enumerate}
        \item No generalized edge group of $\mathbb{G}$ is mapped properly
            into a conjugate of itself by some and hence any automorphism $\Phi$ representing $\varphi$.
        \item Edge groups of $\mathbb{G}$ have finite index in
            their incident vertex groups,
            i.e.~the Bass--Serre tree $\Gamma$ is locally finite
            and there are only finitely many isomorphism types of graphs of groups $\mathcal{G}'$
            homotopy equivalent to $\mathbb{G}$ all of whose edges are surviving.
    \end{enumerate}
    Then there exists a train track map $f\colon \mathcal{G} \to \mathcal{G}$ representing $\varphi$
    on a $\mathbb{G}$-marked graph of groups $\mathcal{G}$ with the appropriate property above.
\end{thm}

(The finite generation assumption is unnecessary in this section.)

The broad-strokes outline of the proof of \Cref{traintrackthm} is much the same as the previous example.
By folding at illegal turns, we often produce nontrivial tightening,
which decreases the Perron--Frobenius eigenvalue.
By controlling the presence of valence-one and valence-two vertices,
we may argue that the transition matrix lies in a finite set of matrices,
thus the Perron--Frobenius eigenvalue may only be decreased finitely many times.
In the remainder of this section,
we make this precise by recalling Bestvina and Handel's original analysis.
The proofs are largely identical to the original, so we omit them.

\paragraph{Subdivision.}
Given a topological representative $f\colon \mathcal{G} \to \mathcal{G}$,
if $p$ is a point in the interior of an edge $e$
such that $f(p)$ is a vertex, we may give $\mathcal{G}$ a new graph of groups structure
by declaring $p$ to be a vertex
with vertex group equal to $\mathcal{G}_e$.
If $f(e) = \gamma_1 g\gamma_2$ is the subdivision of the graph of groups edge path $f(e)$
at the image of the point $p$, where $g \in \mathcal{G}_{f(p)}$,
and the new edges incident to $p$ are $e_1$ and $\bar e_2$,
define for definiteness $f(e_1) = \gamma_1g$ and $f(e_2) = \gamma_2$.
The map $f_p \colon \mathcal{G}_p \to \mathcal{G}_{f(p)}$
is given by the following commutative diagram
\[  \begin{tikzcd}[column sep = large]
    & \mathcal{G}_e \ar[ld, "f_{e'}"'] \ar[rd, "f_{e''}"] & \\
    \mathcal{G}_{e'} \ar[r, "\iota_{e'}"'] & \mathcal{G}_{f(p)} & \mathcal{G}_{e''},
    \ar[l, "\operatorname{ad}(g)\iota_{\bar e''}"] \\
\end{tikzcd}\]
where $e'$ and $e''$ are the last edge of $\gamma_1$ and first edge of $\gamma_2$, respectively.
(Note that up to homotopy, we may factor $g \in \mathcal{G}_{f(p)}$ arbitrarily
as $g'g''$ and define $f(e_1) = \gamma_1 g'$ and $f(e_2) = g''\gamma_2$.)

\begin{lem}[Lemma 1.10 of \cite{BestvinaHandel}]
	If $f\colon \mathcal{G} \to \mathcal{G}$ is a topological representative
	and $f_1\colon \mathcal{G}_1 \to \mathcal{G}_1$
	is obtained by subdivision,
	then $f_1$ is a topological representative.
	If $f$ is irreducible, then $f_1$ is too,
	and the associated Perron--Frobenius eigenvalues are equal.\hfill\qedsymbol
\end{lem}

\paragraph{Valence-One Homotopy.}
Recall that a valence-one vertex $v$ with incident edge $e$ is \emph{inessential}
if the monomorphism $\iota_e\colon\mathcal{G}_e \to \mathcal{G}_v$ is an isomorphism.

If $v$ is an inessential valence-one vertex with incident edge $e$,
let $\mathcal{G}_1$ denote the subgraph of groups determined by $G \setminus \{e,v\}$,
and let $\pi\colon \mathcal{G} \to \mathcal{G}_1$ be the map collapsing $e$.
Let $f_1\colon \mathcal{G}_1 \to \mathcal{G}_1$ be the topological representative
obtained from $\pi f|_{\mathcal{G}_1}$ by tightening and collapsing a maximal pretrivial forest.
We say that $f_1\colon \mathcal{G}_1 \to \mathcal{G}_1$ is obtained from $f\colon \mathcal{G} \to \mathcal{G}$
by a \emph{valence-one homotopy.}

\begin{lem}[Lemma 1.11 of \cite{BestvinaHandel}]
	If $f\colon \mathcal{G} \to \mathcal{G}$ is an irreducible topological representative
	with Perron--Frobenius eigenvalue $\lambda$ and $f_1\colon \mathcal{G}_1 \to \mathcal{G}_1$
	is obtained from $f\colon \mathcal{G} \to \mathcal{G}$
	by performing valence-one homotopies on all inessential valence-one vertices of $\mathcal{G}$
	followed by the collapse of a maximal invariant forest,
	then $f_1\colon \mathcal{G}_1 \to \mathcal{G}_1$ is irreducible,
	and the associated Perron--Frobenius eigenvalue $\lambda_1$ satisfies $\lambda_1 < \lambda$.
\end{lem}

\paragraph{Valence-Two Homotopy.}
We likewise distinguish two kinds of valence-two vertices.
A valence-two vertex $v$ with incident edges $e_i$ and $e_j$ is \emph{inessential}
if at least one of the monomorphisms
$\iota_{e_i}\colon \mathcal{G}_{e_i} \to \mathcal{G}_v$ and
$\iota_{e_j}\colon \mathcal{G}_{e_j} \to \mathcal{G}_v$ is an isomorphism,
say $\iota_{e_j}\colon \mathcal{G}_{e_j} \to \mathcal{G}_v$.
Let $\pi$ be the map that collapses $e_j$ to a point and expands $e_i$ over $e_j$.
Define a map $f'\colon \mathcal{G} \to \mathcal{G}$ by tightening $\pi f$.
Observe that no vertex of $\mathcal{G}$ is mapped to $v$.
Thus we may define a new graph of groups structure $\mathcal{G}'$
by removing $v$ from the set of vertices.
Thus the edge path $e_i\bar e_j$ is now an edge,
which we will call $e_i$ with edge group $\mathcal{G}_{e_i}$.
Let $f''\colon \mathcal{G}' \to \mathcal{G}'$ be the map obtained by tightening
$f''(e_i)  = f'(e_i\bar e_j)$.
Finally, let $f_1 \colon \mathcal{G}_1 \to \mathcal{G}_1$
be the topological realization obtained by collapsing a maximal pretrivial forest.
We say that $f_1\colon \mathcal{G}_1 \to \mathcal{G}_1$ is obtained by a 
\emph{valence-two homotopy of $v$ across $e_j$.}

\begin{lem}[Lemma 1.13 of \cite{BestvinaHandel}]
	Let $f\colon \mathcal{G} \to \mathcal{G}$ be an irreducible topological representative,
	and suppose $\mathcal{G}$ has no inessential valence-one vertices.
	Suppose $f_2\colon \mathcal{G}_2 \to \mathcal{G}_2$ 
	is the irreducible topological representative obtained
	by performing a valence-two homotopy of $v$ across $e_j$
	followed by the collapse of a maximal invariant forest.
	Let $M$ be the transition matrix of $f$ and choose a 
	positive eigenvector $\vec w$ with $M\vec w = \lambda\vec w$.
	If $w_i \le w_j$, then $\lambda_2 \le \lambda$;
	if $w_i < w_j$, then $\lambda_2 < \lambda$.
\end{lem}

\begin{rk}
	The statement of the lemma hides a problem:
	if we cannot freely choose which edge incident to an inessential valence-two
	vertex to collapse via a valence-two homotopy,
	we may be forced to \emph{increase} $\lambda$.
    Since we aim always to decrease $\lambda$,
    we cannot perform such valence-two homotopies.
    These are the \emph{problematic valence-two vertices} mentioned in the introduction.
    Our assumptions are designed to limit their proliferation.
\end{rk}

\paragraph{Folding.}
Suppose some pair of edges $e_1$, $e_2$ in $\mathcal{G}$ sharing a common initial vertex
have the same $f$-image 
(as graph-of-groups edge paths).
Define a new graph of groups $\mathcal{G}_1$ by identifying $e_1$ and $e_2$ to a single edge $e$.
The map $f\colon \mathcal{G} \to \mathcal{G}$ descends to a well-defined homotopy equivalence
$f_1\colon \mathcal{G}_1 \to \mathcal{G}_1$.
This is an \emph{elementary fold.}
More generally if $e_1'$ and $e_2'$ are maximal initial segments of $e_1$ and $e_2$
with equal $f$-images and endpoints sent to a vertex by $f$,
we first subdivide at the endpoints of $e_1'$ and $e_2'$ if they are not already vertices
and then perform an elementary fold on the resulting edges.

Let us remark that when lifting the map $f$ to the Bass--Serre tree, 
it is possible that the lifted map may identify a pair of edges 
$\tilde e$ and $g.\tilde e$ in the same orbit $e$
and sharing a common initial vertex.
Suppose $\tau(e) = v$ and that $e'$ is the last edge in the edge path $f(e)$.
This happens when $\mathcal{G}_v$ contains an element $g$ such that 
$g$ is not in the image
$\iota_e(\mathcal{G}_e)$ but $f_v(g)$ is in the image of $\iota_{e'}(\mathcal{G}_{e'})$.
To perform the fold, we may need to subdivide at the preimage of $\tau(\bar e')$.
In $\mathcal{G}$, the ``fold'' merely changes the edge group,
increasing it to $f_v^{-1}(f_v(\mathcal{G}_v) \cap \iota_{e'}(\mathcal{G}_{e'}))$.
This fold has no effect on the transition matrix for $f$
unless nontrivial tightening occurs,
in which case the Perron--Frobenius eigenvalue decreases.
(See \cite[Remark 1.6]{BestvinaHandel}.)
It is these folds which may introduce problematic valence-two vertices,
so we shall have to be careful about when to perform them.

\begin{lem}[Lemma 1.15 of \cite{BestvinaHandel}]
	Suppose $f\colon \mathcal{G} \to \mathcal{G}$ is an irreducible topological representative
	and that $f_1\colon \mathcal{G}_1 \to \mathcal{G}_1$ is obtained by folding a pair of edges.
	If $f_1$ is a topological representative, then it is irreducible,
	and the associated Perron--Frobenius eigenvalues satisfy $\lambda_1 = \lambda$.
	Otherwise, let $f_2 \colon \mathcal{G}_2 \to \mathcal{G}_2$
	be the irreducible topological representative obtained by tightening,
	collapsing a maximal pretrivial forest, and collapsing
	a maximal invariant forest.
	Then the associated Perron--Frobenius eigenvalues satisfy $\lambda_2 < \lambda$.
\end{lem}

\begin{lem}
    \label{foldingwithassumption1}
    Suppose $f\colon \mathcal{G} \to \mathcal{G}$ is an irreducible topological representative
    of the outer automorphism $\varphi \in \out(\pi_1(\mathcal{G}))$
    and no generalized edge group of $\mathcal{G}$
    is mapped properly into a conjugate of itself by some
    and hence any automorphism $\Phi$ representing $\varphi$.
    Then the edge groups of $\mathcal{G}$ are all isomorphic,
    and the injective maps $f_{e,e_i}$ are isomorphisms.
    If $f(e) = f(e')$ and $e$ and $e'$ share a common vertex so that we may fold $e$ and $e'$,
    then $\iota_e(\mathcal{G}_e) = \iota_{e'}(\mathcal{G}_{e'})$
    as subgroups of $\mathcal{G}_{\tau(e)}$,
    and the resulting map $f' \colon \mathcal{G}' \to \mathcal{G}'$
    again has the property above.
\end{lem}

\begin{proof}
    It is clear that for each edge $e$ and each edge $e_i$ of the edge path $f(e)$,
    the map $f_{e,e_i}$ is injective,
    for if it were not, then the map on the fundamental group $f_\sharp$ would fail to be injective.
    Since $f$ is irreducible,
    there exists a sequence $e = e^0,e^1,\ldots,e^k = e$
    such that $e^{i}$ appears in the edge path $f(e^{i-1})$ in either orientation
    for $i$ satisfying $1 \le i \le k$.
    Thus we have a sequence of injective homomorphisms
    \[  \begin{tikzcd}
        \mathcal{G}_{e} = \mathcal{G}_{e^0} \ar[r, "f_{e^0,e^1}"] & \cdots 
        \ar[r, "f_{e^{k-1},e^k}"] & \mathcal{G}_{e^k} = \mathcal{G}_e.
    \end{tikzcd}\]
    The composition of these maps is an isomorphism,
    for otherwise the generalized edge group corresponding to $\mathcal{G}_e$
    would be mapped properly into a conjugate of itself,
    from which we conclude that each composing map is an isomorphism.
    By irreducibility of $f$, we may choose $e^1$ to be any edge of the edge path $f(e)$,
    so we see that each map $f_{e,e_i}$ is an isomorphism.

    Suppose that $f(e) = f(e') = g_0e_1g_1\ldots e_kg_k$.
    We have the following pair of commutative diagrams
    \[  \begin{tikzcd}
        \mathcal{G}_e \ar[r, "\iota_e"] \ar[d, "f_{e,e_k}"'] 
        & \mathcal{G}_v \ar[d, "\operatorname{ad}(g_k)f_v"] 
        & & \mathcal{G}_{e'} \ar[r, "\iota_{e'}"] \ar[d, "f_{e',e_k}"'] 
        & \mathcal{G}_v \ar[d, "\operatorname{ad}(g_k)f_v"] \\
        \mathcal{G}_{e_k} \ar[r, "\iota_{e_k}"]
        & \mathcal{G}_{f(v)}
        & & \mathcal{G}_{e_k} \ar[r, "\iota_{e_k}"]
        & \mathcal{G}_{f(v)}.
    \end{tikzcd}\]
    Since $f_{e,e_k}$ and $f_{e',e_k}$ are isomorphisms,
    we conclude that 
    \[
        \operatorname{ad}(g_k)f_v\iota_e(\mathcal{G}_e) = \iota_{e_k}(\mathcal{G}_{e_k}) 
        = \operatorname{ad}(g_k)f_v\iota_{e'}(\mathcal{G}_{e'})
    \]
    and since $\operatorname{ad}(g_k)f_v$ is injective,
    it follows that $\iota_e(\mathcal{G}_e) = \iota_{e'}(\mathcal{G}_{e'})$.
    Observe that if $\Gamma$ is the original Bass--Serre tree
    and $\Gamma'$ is the new Bass--Serre tree after  folding,
    then every edge stabilizer in $\Gamma$ is an edge stabilizer in $\Gamma'$
    and conversely every edge stabilizer in $\Gamma'$ comes from an edge stabilizer in $\Gamma$.
    It follows that $\mathcal{G}$ and $\mathcal{G}'$ have the same
    generalized edge groups, so the resulting map $f'\colon \mathcal{G}' \to \mathcal{G}'$
    has the desired property.
\end{proof}

From \Cref{foldingwithassumption1},
we deduce that in fact, if $\mathcal{G}$ satisfies our first standing assumption, then
all folding takes place between \emph{distinct} edges of $\mathcal{G}$,
and that the edge group of the newly folded edge is isomorphic to $\mathcal{G}_{e}$ and $\mathcal{G}_{e'}$.

\begin{lem}
    \label{edgesbound}
    The number of edges of a graph of groups homotopy equivalent to $\mathbb{G}$
    without inessential valence-one or valence-two vertices is bounded.

    If the Bass--Serre tree of $\mathbb{G}$ is locally finite
    and there are finitely many isomorphism types of graphs of groups homotopy equivalent to $\mathbb{G}$
    all of whose edges are surviving,
    then there is a bound to the number of edges of a graph of groups homotopy equivalent to $\mathbb{G}$
    without inessential valence-one vertices and for which every inessential valence-two vertex
    is problematic.
\end{lem}

Here we call an inessential valence-two vertex of a graph of groups \emph{problematic}
if exactly one of its edge-to-vertex group inclusions is surjective.

\begin{proof}
    Call a vertex of a graph of groups $\mathcal{G}$ \emph{essential} if for all oriented edges $e \in \st(v)$,
    the monomorphism $\iota_{\bar e} \colon \mathcal{G}_e \to \mathcal{G}_v$ is not surjective.
    Because the graph of groups $\mathbb{G}$ is assumed to be reduced,
    every vertex of $\mathbb{G}$ is either essential or incident to an edge $e$ which forms a loop
    and for which one of the monomorphisms $\iota_e$ or $\iota_{\bar e}$ is surjective.
    Let $\eta(\mathbb{G})$ be the number of essential vertices of $\mathbb{G}$
    and let $\beta(\mathbb{G})$ be the first Betti number of $\mathbb{G}$.

    We claim that any graph of groups $\mathcal{G}$ homotopy equivalent to $\mathbb{G}$
    has at most $\eta(\mathbb{G})$ essential vertices, but it may have fewer.
    Each essential vertex of $\mathbb{G}$ corresponds to the conjugacy class
    of a \emph{maximal elliptic} subgroup $H$;
    here elliptic means $H$ stabilizes some vertex of the relevant Bass--Serre tree
    and maximal means that $H$ is not conjugate to a proper subgroup of an elliptic subgroup.
    Let $T$ be the Bass--Serre tree of $\mathbb{G}$.
    We claim that each maximal elliptic subgroup $H$ of $F$ with the additional property
    that the fixed-point set of $H$ is bounded 
    is represented by the vertex group of some essential vertex of $\mathbb{G}$.
    Indeed, if $H$ is maximal elliptic,
    then it is contained in and hence equal to some vertex stabilizer in $T$.
    The corresponding vertex $v$ of $\mathbb{G}$ must be essential:
    by maximality, if $v$ is incident to an edge which forms a loop
    and for which one of the monomorphisms $\iota_e$ or $\iota_{\bar e}$ is surjective,
    then both are surjective, which would contradict boundedness of the fixed-point set of $H$.
    The properties of being maximal elliptic and having bounded fixed-point set
    are invariant under homotopy equivalence by \cite[Theorem 3.8]{GuirardelLevittDeformation},
    from which it follows that $\mathcal{G}$ has at most $\eta(\mathbb{G})$ essential vertices.

    Now, $\mathcal{G}$ has at most $2\eta(\mathbb{G}) + 3\beta(\mathbb{G}) - 3$ edges.
    To see this, form a new graph $G'$ from $G$
    by cyclically ordering the essential vertices of $\mathcal{G}$
    and attaching an edge from each essential vertex to its neighbors in the cyclic ordering.
    The graph $G'$ has no valence-one or valence-two vertices
    and first Betti number at most $\eta(\mathbb{G}) + \beta(\mathbb{G})$.
    An Euler characteristic argument reveals that $G'$ 
    has at most $3(\eta(\mathbb{G}) + \beta(\mathbb{G})) - 3$ edges,
    from which the stated bound for $\mathcal{G}$ follows.
    
    Now suppose that the Bass--Serre tree of $\mathbb{G}$ is locally finite
    and that there are finitely many isomorphism types of graphs of groups
    homotopy equivalent to $\mathbb{G}$ all of whose edges are surviving.
    The Bass--Serre tree of each of these graphs of groups is locally finite
    with finitely many orbits of vertices,
    so there is a maximum valence $M$ of any vertex in any such Bass--Serre tree.
    Let $\mathcal{G}$ be a graph of groups homotopy equivalent to $\mathbb{G}$
    without inessential valence-one vertices
    and for which every inessential valence-two vertex is problematic.
    It is not quite true that the graph of groups obtained from $\mathcal{G}$
    by performing a maximal collapse of edges incident to inessential valence-two vertices
    has the property that every edge is surviving,
    so we further collapse all non-surviving edges
    to obtain a graph of groups $\mathcal{G}'$ without inessential valence-one or valence-two vertices,
    every edge of which is surviving.
    Let $\Gamma$ and $\Gamma'$ be the Bass--Serre trees of $\mathcal{G}$ and $\mathcal{G}'$.
    The collapse map $p\colon \Gamma \to \Gamma'$ has compact fibers,
    so each vertex $\tilde v$ of $\Gamma'$ corresponds to a finite subtree $T_{\tilde v}$ of $\Gamma$.
    The valence of $\tilde v$ is the sum of the edges incident to $T_{\tilde v}$ but not contained in it
    and is bounded by $M$. 
    Since each valence-two vertex of $\mathcal{G}$ is problematic,
    each vertex of $\Gamma$ has valence at least three,
    and it follows that the size of each tree $T_{\tilde v}$ is bounded depending only on $M$,
    from which we conclude that there is a bound 
    on the number of problematic valence-two vertices of $\mathcal{G}$
    and thus a bound on the number of edges of $\mathcal{G}$.
\end{proof}

\begin{proof}[Proof of \Cref{traintrackthm}]
	Let $f\colon \mathbb{G} \to \mathbb{G}$ be an irreducible topological representative of $\varphi$,
    where we recall that $\mathbb{G}$ is \emph{reduced} in the sense of \cite{Forester}
    and satisfies one of our assumptions.

	Suppose the Perron--Frobenius eigenvalue $\lambda$ satisfies $\lambda =1$.
	Then $f$ transitively permutes the edges of $G$ and is thus a train track map.
	So assume $\lambda > 1$.
    Recall our standing assumptions:
    \begin{enumerate}
        \item No generalized edge group of $\mathbb{G}$
            is mapped properly into a conjugate of itself by some
            and hence any automorphism $\Phi$ representing $\varphi$.
        \item Edge groups of $\mathbb{G}$ have finite index in the incident vertex groups.
    \end{enumerate}
    If $\mathbb{G}$ satisfies the first assumption,
    we will show that \Cref{foldingwithassumption1}
    implies that any graph of groups $\mathcal{G}'$ obtained from $\mathbb{G}$
    has no problematic valence-two vertices
    and by \Cref{edgesbound} 
    we conclude that there is a uniform bound $L$ to the number of edges of $\mathcal{G}'$.
    If $\mathbb{G}$ satisfies the second assumption,
    then we cannot prevent $\mathcal{G}'$ from having problematic valence-two vertices,
    but again by \Cref{edgesbound}
    there is a uniform bound $L$ to the number of edges.

	We will show that if $f\colon \mathbb{G} \to \mathbb{G}$ is not a train track map,
	then there is an irreducible topological representative 
    $f_1 \colon \mathcal{G}_1 \to \mathcal{G}_1$
	without inessential valence-one vertices
	such that the associated Perron--Frobenius eigenvalues satisfy $\lambda_1 < \lambda$.
    If $\mathbb{G}$ satisfies the first assumption,
    we show that $\mathcal{G}_1$ has no inessential valence-two vertices.
    If instead $\mathbb{G}$ satisfies the second assumption,
    we show that inessential valence-two vertices of $G_1$ are problematic.
    It follows that the size of the transition matrix of $f_1$
	is uniformly bounded.

	Furthermore, if $M$ is an irreducible matrix,
	its Perron--Frobenius eigenvalue $\lambda$ is bounded below by the minimum sum of the entries of a row of $M$.
	To see this, let $\vec w$ be a positive eigenvector.
	If $w_j$ is the smallest entry of $\vec w$, $\lambda w_j = (M\vec w)_j$
	is greater than $w_j$ times the sum of the entries of the $j$th row of $M$.

	Thus if we iterate this argument reducing the Perron--Frobenius eigenvalue,
	there are only finitely many irreducible transition matrices that can occur,
	so at some finite stage the Perron--Frobenius eigenvalue will reach a minimum.
	At this point, we must have a train track map.

	To complete the proof, we turn to the question of decreasing $\lambda$.
	Suppose $f\colon \mathbb{G} \to \mathbb{G}$ is not a train track map.
	Then there exists a point $p$ in the interior of an edge such that $f(p)$ is a vertex,
	and $f^k$ is not locally injective (as a map of graph of groups) at $p$ for some $k > 1$.
	We assume that topological representatives act linearly on edges 
    with respect to some metric on $G$.
	Since $\lambda > 1$,
    this means the set of points of $G$ eventually mapped to a vertex is dense.
	Thus we can choose a neighborhood $U$ of $p$ so small 
    that it satisfies the following conditions.
	\begin{enumerate}
		\item The boundary $\partial U$ is a two-point set $\{ s,t\}$,
			where $f^\ell(s)$ and $f^\ell(t)$ are vertices for some $\ell \ge 1$.
        \item $f^{i}|_U$ is injective (as a map of graphs of groups) for $1 \le i \le k-1$.
		\item $f^k$ is two-to-one on $U\setminus\{p\}$,
            and $f^k(U)$ is contained within a single edge.
		\item $p \notin f^i(U)$, for $1 \le i \le k$.
	\end{enumerate}

    Note that a priori the map $f^i|_U$ could fail to be injective as a map of graphs
    sooner than it fails to be injective as a map of graphs of groups.
    Suppose that there exists $1 \le j \le k$ such that $f^i|_U$ is injective 
    as a map of graphs
    for $1 \le i \le j-1$, and $f^j|_U$ 
    is two-to-one on $U\setminus\{p\}$ as a map of graphs.
    We will show that \Cref{foldingwithassumption1}
    implies that in fact $j = k$.
    If instead edge groups of $\mathbb{G}$ merely have finite index
    in the incident vertex groups, we allow the same-orbit edge fold,
    in view of our bound on the number of problematic valence-two vertices that can arise.

	First we subdivide at $p$. 
    Then we subdivide at $f^i(s)$ and $f^i(t)$ for $0 \le i \le \ell - 1$
	(in reverse order so that subdivision is allowed). 
    The vertex $p$ has valence two;
	denote the incident edges by $e$ and $e'$.
	Observe that in the Bass--Serre tree $\Gamma$,
    there are lifts $\tilde e$ and $\tilde e'$ such that
    $\tilde f^{k-1}(\tilde e)$ and $\tilde f^{k-1}(\tilde e')$ 
    are single edges sharing a common initial vertex that are identified by $\tilde f$.
	Thus we may fold.
    Suppose that $j < k$.
    Then this fold increases the edge group 
    $\mathcal{G}_{f^{k-1}(e)} = \mathcal{G}_{f^{k-1}(e')}$,
    so the map of edge groups 
    $\mathcal{G}_{f^{k-1}(e)} \to \mathcal{G}_{f^k(e)}$ is injective 
    (for otherwise $f_\sharp$ could not be injective) and not surjective.
    This contradicts \Cref{foldingwithassumption1}.
    Therefore $j=k$ if $\mathbb{G}$ satisfies the first assumption,
    and in fact $f^{k-1}(e)$ and $f^{k-1}(e')$ 
    are distinct single edges that are identified by $f$.

	The resulting map $f'\colon \mathcal{G}' \to \mathcal{G}'$ 
    may be a topological representative, 
	in which case the Perron--Frobenius eigenvalue $\lambda'$ 
    satisfies $\lambda' = \lambda$.
	In this case $\tilde f'^{k-2}(\tilde e)$ and $\tilde f'^{k-2}(\tilde e')$
    are single edges that are identified by $f$.
	In the contrary case, nontrivial tightening occurs.
	After collapsing a maximal pretrivial forest and a maximal invariant forest,
	the resulting irreducible topological representative 
    $f''\colon \mathcal{G}'' \to \mathcal{G}''$
	has Perron--Frobenius eigenvalue $\lambda''$ satisfying $\lambda'' < \lambda$.

	Repeating this dichotomy $k$ times if necessary,
	we have either decreased $\lambda$,
	or we have folded $e$ and $e'$ (which are distinct edges in the graph of groups)
    so that $p$ is now an inessential valence-one vertex.

	We remove inessential valence-one and non-problematic valence-two vertices 
    by the appropriate homotopies.
    Note that if $\mathbb{G}$ and $\varphi$ satisfy the first assumption,
    all valence-two vertices present were created by subdivision,
    not same-orbit folding,
    and thus are not problematic.
	Since valence-one homotopy always decreases the Perron--Frobenius eigenvalue,
	the resulting irreducible topological representative 
    $f_1\colon \mathcal{G}_1 \to \mathcal{G}_1$
	has Perron--Frobenius eigenvalue $\lambda_1$ satisfying $\lambda_1 < \lambda$.
\end{proof}

\begin{rk}
	As in the original, the proof of \Cref{traintrackthm} provides in outline
	an algorithm that takes as input a topological representative of an irreducible outer automorphism
	and returns a train track map.
    To make it a true algorithm in general,
    one needs an ``oracle'' that can compute images of the various homomorphisms
    $f_e$ and $f_v$,
    compute products of elements in the vertex groups $\mathcal{G}_v$
    and tell when two vertex group elements are equal.
\end{rk}

A \emph{reduction} for an outer automorphism $\varphi\in\out(\pi_1(\mathcal{G}))$
is a topological representative $f\colon \mathcal{G} \to \mathcal{G}$
which has no inessential valence-one vertices and no invariant forests
but has a nontrivial invariant subgraph.
If $\varphi$ has a reduction, then it is \emph{reducible}---i.e.~not irreducible.
Let $F = A_1*\dotsb*A_n*F_k$ be a free product,
represented as the fundamental group of a graph of groups $\mathbb{G}$
with trivial edge groups, vertex groups  
the $A_i$ and ordinary fundamental group free of rank $k$.
Define the \emph{complexity} of $F$ 
relative to $\mathbb{G}$ to be the quantity $n + 2k - 1$.
If $F'$ is a free factor of $F$ relative to $\mathbb{G}$
we may define the complexity of $F'$ relative to $\mathbb{G}$ analogously.
The final result of this section is the following characterization of reducibility
for outer automorphisms $\varphi$
represented on $\mathbb{G}$-marked graphs of groups with trivial edge groups.

\begin{prop}
	\label{reduciblefreeproduct}
    Let $F$ be a free product. An outer automorphism $\varphi$
    represented on a $\mathbb{G}$-marked graph of groups with trivial edge groups
    is reducible relative to $\mathbb{G}$ 
    if and only if there are free factors $F^1,\dotsc,F^m$ of $F$ with
	positive complexity such that $F^1*\dotsb*F^m$ is a free factor of $F$
	and $\varphi$ cyclically permutes the conjugacy classes of the $F^i$.
\end{prop}

\begin{proof}
    Suppose first that $\varphi$ is reducible relative to $\mathscr{A}$; 
	let $f\colon\mathcal{G} \to \mathcal{G}$ be a reduction
	and let $G_i = f^i(G_1)$, $0 \le i \le m-1$ denote
	distinct noncontractible components of an $f$-invariant subgraph.
	Then each $\pi_1(\mathcal{G}|_{G_i})$
	determines a free factor $F^i$
	with positive complexity such that $F^1*\dotsb*F^m$ is a free factor of $F$
	and such that $\varphi$ cyclically permutes the conjugacy classes of the $F^i$.

	Conversely, suppose $F^1,\dotsc,F^m$ are free factors with positive complexity
	as in the statement of the proposition. Take $F^{m+1}$ a free factor so that
	$F = F^1*\dotsb*F^m*F^{m+1}$. Suppose that $n_i$ and $k_i$ are the data determining the complexity of $F^i$
	for $1 \le i \le m+1$.
	Let $\mathbb{G}_i$ be the thistle with $n_i$ prickles and $k_i$ petals 
	(if $n_{m+1}=k_{m+1} = 0$, then $\mathbb{G}_{m+1}$ is a vertex)
	and distinguished vertex $\star_i$.
	For each $i$ satisfying $1 \le i \le m$ choose automorphisms $\Phi_i \colon F \to F$
	representing $\varphi$ such that $\Phi_i(F^i) = F^{i+1}$, with indices taken mod $m$,
	and let $f_i\colon\mathbb{G}_i \to \mathbb{G}_{i+1}$ be the corresponding
	topological representatives
	taking $\star_i$ to $\star_{i+1}$.
	Define $\mathcal{G}$ to be the union of the $\mathbb{G}_i$ for $1 \le i \le m+1$
	together with, for $1 \le i \le m$, an oriented edge $E_i$
	connecting $\star_i$ to $\star_{m+1}$.

	Collapsing the $E_i$ to a point yields a homotopy equivalence $\mathcal{G} \to \mathbb{G}$,
	where $\mathbb{G}$ is the thistle with $n$ prickles and $k$ petals.
	Identifying the image of $\pi_1(\mathbb{G}_i,\star_i)$ with $F^i$
	will serve as (the inverse of) a marking.
	We will use $\Phi_1$ to create 
	a topological representative $f\colon\mathcal{G} \to \mathcal{G}$ for $\varphi$.
	Define $f(\mathbb{G}_i) = f_i(\mathbb{G}_i)$ for $1 \le i \le m$.
	By assumption there exist $c_i \in F$ such that $\Phi_1(x) = c_i\Phi_i(x)c_i^{-1}$.
	Choose $\gamma_i$ a closed tight edge path based at $\star_{m+1}$
	representing $c_i$ (so $\gamma_1$ is the trivial path)
	and define $f(E_i) = \gamma_iE_{i+1}$ with indices taken mod $m$.
    Finally define $f(\mathbb{G}_{m+1})$ by $\Phi_1$ and the marking on $\mathbb{G}_{m+1}$.

	The topological representative $f\colon\mathcal{G} \to \mathcal{G}$
	is a reduction for $\varphi$ unless $\mathcal{G}$ has an invariant contractible forest.
	Since thistles have contractible subgraphs, there are a few possibilities.
	If there is a family of non-loop edges $e_1,\dotsc,e_m$ with $e_i \in \mathbb{G}_i$
    and $f(e_i) = e_{i+1}$ with indices mod $m$, we may collapse each of these edges.
    (Note that up to equivalence, if $f(e_i) = e_{i+1}$ as a map of graphs,
    then $f(e_i) = e_{i+1}$ as a map of graphs of groups.)
	Likewise if some non-loop edge of $\mathbb{G}_{m+1}$ is sent to itself, we may collapse it.
	If each $c_i = 1 \in F$, then the $E_i$ also form an invariant forest
	that is contractible if the subgraph they span 
	contains at most one vertex with vertex group some $A_i$.
	After all these forest collapsings, the only worry is that $F^{m+1}$
	has nonpositive complexity and the $E_i$ would be collapsed, leaving $G$
	as the only $f$-invariant subgraph.
	In this case, choose $A$ an edge of $\mathbb{G}_1$ sharing an initial vertex with $E_1$,
	and change $f$ via a homotopy with support in $E_1$ so that
	$f(E_1) = f(A)f(\bar A)E_2$, then fold the initial segment of $E_1$ mapping to $f(A)$ with all of $A$.
	The resulting graph is combinatorially identical to $\mathcal{G}$ but the markings differ.
	Now $f(E_1) = f(\bar A)E_2$ and $f(E_k) = \bar AE_1$, 
	so the $E_i$ no longer form an invariant forest.
\end{proof}

\section{Relative Train Track Maps}
\label{relativetraintracksection}
The purpose of this section is to prove the general case of \Cref{relativetraintrack}.
The strategy is to adapt arguments in \cite[Section 5]{BestvinaHandel}
and \cite[Section 2]{FeighnHandelAlg}.

\paragraph{Filtrations.}
A \emph{filtration} on a marked graph of groups $\mathcal{G}$
with respect to a topological representative $f\colon \mathcal{G} \to \mathcal{G}$
is an increasing sequence $\varnothing = G_0 \subset G_1 \subset \dotsb\subset G_m = G$
of $f$-invariant subgraphs. The subgraphs are not required to be connected.

\paragraph{Strata.}
The \emph{$r$th stratum} of $\mathcal{G}$ 
is the subgraph $H_r$ containing those edges of $G_r$ not contained in $G_{r-1}$.
An edge path has \emph{height $r$} if it is contained in $G_r$
and meets the interior of $H_r$.
If both edges of a turn $T$ are contained in a stratum $H_r$,
then $T$ is a \emph{turn in $H_r$.} If a path has height $r$
and contains no illegal turns in $H_r$ then it is \emph{$r$-legal.}

When we think of a stratum $H_r$ or a filtration element $G_r$ as a graph of groups in its own right,
the vertex and edge groups of $H_r$ and $G_r$ are equal to what they are in $\mathcal{G}$.
In the language of Bass \cite{Bass}, we work with subgraphs of groups,
not subgraphs of subgroups.

\paragraph{Transition Submatrices.}
Relabeling the edges of $G$ and thus permuting the rows and columns
of the transition matrix $M$ so that the edges of $H_i$ precede those of $H_{i+1}$,
$M$ becomes block upper-triangular, with the $i$th block $M_i$ equal to the square submatrix
of $M$ containing those rows and columns corresponding to edges in $H_i$.

A filtration is \emph{maximal} when each $M_i$ is either irreducible or the zero matrix.
If $M_i$ is irreducible, 
call $H_i$ an \emph{irreducible stratum} and a \emph{zero stratum} otherwise.
If $H_i$ is irreducible, $M_i$ has an associated Perron--Frobenius eigenvalue $\lambda_i \ge 1$.
If $\lambda_i > 1$, then $H_i$ is an \emph{exponentially growing stratum.}
Otherwise $\lambda_i = 1$, we say $H_i$ is \emph{non-exponentially growing}
and $M_i$ is a transitive permutation matrix.

Associated to a topological representative $f\colon \mathcal{G} \to \mathcal{G}$
there is a maximal filtration $\varnothing = G_0 \subset \dotsb \subset G_m = G$
defined as follows \cite[p.~33]{BestvinaHandel}.
Order the edges of $G$, and let
$M = (m_{ij})$ be the resulting transition matrix for $f$.
Construct a graph $E$ with a vertex $v_i$ for each edge $e_i$ of $G$,
and $m_{ij}$ oriented edges from $v_j$ to $v_i$.
Two edges $e_1$ and $e_2$ belong to the same irreducible stratum
if there exists an oriented path from $v_1$ to $v_2$ and an oriented path from $v_2$ to $v_1$.
An edge $e_1$ does not belong to an irreducible stratum 
if there is no oriented edge path from $v_1$ to itself.
A collection of such edges may determine a zero stratum
if for each pair of edges $e_1$ and $e_2$ in the collection,
there is no oriented edge path from $v_1$ to $v_2$ nor from $v_2$ to $v_1$.
(Perhaps it is easiest to therefore just let each zero stratum be a single edge.)
Let $H$ and $H'$ be two resulting strata;
we define a partial order on strata.
Put $H$ before $H'$
if there are edges $e_1 \in H$ and $e_2 \in H'$
such that there is an oriented path from $v_2$ to $v_1$.
Complete this partial order to a total order arbitrarily.
Thus a maximal filtration associated to $f$ is not unique,
although the irreducible strata are.
We will think of a maximal filtration as part of the data of a topological representative
$f\colon \mathcal{G} \to \mathcal{G}$.

The following lemma is an observation we made in the proof of \Cref{foldingwithassumption1}.

\begin{lem}
    \label{coHopfianisomorphic}
    Let $f\colon\mathcal{G} \to \mathcal{G}$ be a topological representative
    of an outer automorphism $\varphi \in \out(\pi_1(\mathcal{G}))$
    with irreducible stratum $H_r$
    and the property that no iterate of $\Phi$
    maps a generalized edge group of $\mathcal{G}$ properly into a conjugate of itself
    for some and hence any automorphism $\Phi$ representing $\varphi$.
    All edge groups in $H_r$ are isomorphic,
    and in fact if $e_i$ is an edge of $H_r$ in the $f$-image of the edge $e$ of $H_r$,
    then the map $f_{e,e_i}$ is an isomorphism.
\end{lem}

\begin{proof}
    Let $e_1$ and $e_2$ be edges of $H_r$.
    By irreducibility, there is some $k_1 \ge 1$ and $k_2 \ge 1$
    such that the edge path $f^{k_1}(e_1)$ contains $e_2$
    and similarly $f^{k_2}(e_2)$ contains $e_1$.
    This implies that there are injective homomorphisms 
    $\mathcal{G}_{e_1} \to \mathcal{G}_{e_2}$
    and $\mathcal{G}_{e_2} \to \mathcal{G}_{e_1}$.
    The double composition 
    $\mathcal{G}_{e_1}  \to \mathcal{G}_{e_2} \to \mathcal{G}_{e_1}$
    must be an isomorphism,
    so we conclude that each composing homomorphism is an isomorphism.
    In fact, by irreducibility,
    we can arrange so that $f_{e_1,e_i}$ is a composing homomorphism
    of the map $\mathcal{G}_{e_1} \to \mathcal{G}_{e_2}$
    for $e_i$ any edge of $H_r$ contained in the $f$-image  of the edge $e_1$.
\end{proof}

\paragraph{Eigenvalues.}
Let $H_{r_1},\dotsc,H_{r_k}$ be the exponentially growing strata for $f\colon \mathcal{G} \to \mathcal{G}$.
We define $\pf(f)$ to be the sequence of associated Perron--Frobenius eigenvalues
$\lambda_{r_1},\dotsc,\lambda_{r_k}$ in nonincreasing order. We order the set
\[
	\{\pf(f) \mid f\colon\mathcal{G} \to \mathcal{G} \text{ is a topological representative of }\varphi\}
\]
lexicographically; thus if $\pf(f) = \lambda_1,\dotsc,\lambda_k$ and $\pf(f') = \lambda'_1,\dotsc,\lambda'_\ell$,
then $\pf(f) < \pf(f')$ if there is some $j$ with $\lambda_j < \lambda_{j}'$
and $\lambda_i = \lambda'_i$ for $i$ satisfying $1 \le i < j$,
or if $k < \ell$ and $\lambda_i = \lambda'_i$ for $i$ satisfying $1 \le i \le k$.

\paragraph{Relative Train Track Maps.}
Throughout the paper, we will assume our filtrations are maximal unless otherwise specified.
Given $\sigma$ a path in $\mathcal{G}$,
let $f_\sharp(\sigma)$ denote a tight path homotopic rel endpoints to $f(\sigma)$.
(If one wants $f_\sharp(\sigma)$ to be unique, 
one could insist that $f_\sharp(\sigma)$ be in normal form.)
We will denote a maximal filtration preserved by $f\colon \mathcal{G} \to \mathcal{G}$
as $\varnothing = G_0 \subset \dotsb \subset G_m = G$.
The map $f$
is a \emph{relative train track map}
if for every exponentially growing stratum $H_r$, we have
\begin{enumerate}
	\item[\hypertarget{EG-i}{(EG-i)}] Directions in $H_r$ are mapped to directions in $H_r$ by $Df$;
		it follows that every turn with one edge in $H_r$ and the other in $G_{r-1}$ is legal.
	\item[\hypertarget{EG-ii}{(EG-ii)}] If $\sigma \subset G_{r-1}$ is a homotopically nontrivial path
		with endpoints in $H_r \cap G_{r-1}$, 
        then some (and hence every) $f_\sharp(\sigma)$ is nontrivial as well.
	\item[\hypertarget{EG-iii}{(EG-iii)}] If $\sigma\subset G_r$
		is a tight $r$-legal path, then $f(\sigma)$ is an $r$-legal path.
\end{enumerate}

The main result of this section is
\begin{thm}
	\label{relativetraintrackthm}
    Assuming an oracle that can compute products of elements in vertex groups,
    can compute images of injective homomorphisms between edge groups and vertex groups of $\mathcal{G}$
    and can tell when two vertex group elements are equal,
    and that one of the following conditions holds,
	there is an algorithm that takes as input a topological representative
    $f\colon \mathcal{G} \to \mathcal{G}$ of $\varphi \in \out(\pi_1(\mathcal{G}))$
	and improves it to a relative train track map $f'\colon \mathcal{G}' \to \mathcal{G}'$.
    The conditions are as follows.
    \begin{enumerate}
        \item Edge groups of $\mathcal{G}$ are finitely generated
            and for some and hence every $\Phi$ representing $\varphi$,
            no generalized edge group of $\mathcal{G}$ is mapped properly into a conjugate of itself
            by some iterate of $\Phi$.
        \item Vertex groups  of $\mathcal{G}$ are finitely generated
            and edge groups have finite index in the incident vertex groups.
            Additionally there are finitely many isomorphism types of graphs of groups
            $\mathcal{G}'$ homotopy equivalent to $\mathcal{G}$
            with the property that every edge of $\mathcal{G}'$ is surviving.
    \end{enumerate}
\end{thm}

We sketch the outline of the proof:
we begin with a topological representative that is \emph{bounded,}
a term which will be defined below.
We use two new operations, described in \Cref{operation1} and \Cref{operation2}
so that the resulting topological representative satisfies \hyperlink{EG-i}{(EG-i)}
and \hyperlink{EG-ii}{(EG-ii)}.
If \hyperlink{EG-iii}{(EG-iii)} is not satisfied, as in \cite{BestvinaHandel} and \cite{FeighnHandelAlg},
we modify the algorithm in the proof of \Cref{traintrackthm} to reduce $\pf(f)$,
the set of Perron--Frobenius eigenvalues for the exponentially growing strata of $f\colon \mathcal{G} \to \mathcal{G}$,
while remaining bounded.
The boundedness assumption ensures that we will hit a minimum value after a finite number of moves,
at which point \hyperlink{EG-iii}{(EG-iii)} will be satisfied.

Let us say a few words about the input of the algorithm:
a finite, connected graph of groups is a graph
together with groups and homomorphisms between them
(for which we assume we have an oracle).
The data of a topological representative
is the finite connected graph of groups $\mathcal{G}$
together with the filtration, a finite list of subgraphs of $G$,
a list of finite edge  paths $f(e) \in \mathcal{G}$ for each edge $e$ of $G$,
and a finite list of injective homomorphisms
between vertex and edge  groups of $\mathcal{G}$;
one for each $f_v$ and each $f_{e_i}$.
The oracle guarantees that we can, for instance,
tell when two edge paths $f(E)$ and $f(E')$ share a common initial segment
(perhaps after passing to a homotopic topological representative
or changing the marking).

\paragraph{Bounded Representatives.}
As we observed in \Cref{edgesbound},
there exists $L$ such that 
if $\mathcal{G}$ is a marked graph of groups 
without inessential valence-one vertices
and either
\begin{enumerate}
    \item without inessential valence-two vertices, or
    \item which satisfies our second assumption
        and for which every inessential valence-two vertex is problematic,
\end{enumerate}
then $\mathcal{G}$ has at most $L$ edges.
Our first assumption coupled with the assumption that $\varphi$ was irreducible 
allowed us to remove all inessential valence-two vertices that appeared,
but in the general case certain inessential valence-two vertices are useful:
one needs to introduce them so that \hyperlink{EG-i}{(EG-i)} is satisfied, for instance.
As it happens, our method for showing that \hyperlink{EG-ii}{(EG-ii)} is satisfied
may in general even introduce problematic valence-two vertices.
Instead, call a topological representative $f\colon\mathcal{G} \to \mathcal{G}$
\emph{bounded} if there are at most $L$
exponentially growing strata, and if, for each exponentially growing stratum $H_r$,
the associated Perron--Frobenius eigenvalue $\lambda_r$ 
is also the Perron--Frobenius eigenvalue
of a matrix with at most $L$ rows and columns.
As in the proof of \Cref{traintrackthm},
if $f\colon \mathcal{G} \to \mathcal{G}$ is bounded,
the set of $\pf(f')$ for $f'\colon \mathcal{G}' \to \mathcal{G}'$
a bounded representative of $\varphi$ satisfying $\pf(f') \le \pf(f)$ is finite,
so operations decreasing $\pf(f)$ will eventually reach a minimum 
among bounded representatives,
which we will denote $\pfmin$.
Notice as well that the property of being bounded is a property 
of the sequence of numbers $\pf(f)$.

\paragraph{Elementary Moves Revisited.}
In \cite[Lemmas 5.1--5.4]{BestvinaHandel}, Bestvina and Handel
revisit the four elementary moves \emph{subdivision, valence-one homotopy, valence-two homotopy}
and \emph{folding} to analyze their impact on $\pf(f)$.
All of these moves except valence-two homotopy produce a topological representative
$f'\colon \mathcal{G}' \to \mathcal{G}'$ such that the associated Perron--Frobenius eigenvalues satisfy
$\pf(f') \le \pf(f)$.

Let us discuss valence-two homotopy.
Following \cite[p. 35]{BestvinaHandel},
suppose $e_i \in H_i$ and $e_j \in H_j$ are the edges incident to a valence-two vertex $v$.
We assume $i \le j$.
If $i = j$ and $H_i$ is exponentially growing,
choose $i$ and $j$ so that the eigenvector coefficient of $e_i$ is greater than or equal to
that of $e_j$.
Here is the key point: in all cases we perform the valence-two homotopy across $e_i$.
Call such a valence-two homotopy \emph{performable} if after making these choices,
we have that the inclusion $\iota_{e_i}\colon \mathcal{G}_{e_i}\to \mathcal{G}_v$ is an isomorphism.

\begin{lem}
    Suppose we are in the situation of the first assumption.
    All valence-two homotopies are performable,
    perhaps after rearranging strata.
\end{lem}

\begin{proof}
    We continue to use the notation above.
    Suppose at first that $H_i$ is a zero stratum.
    Then the restriction of $f$ to $G_i$ is a homotopy equivalence of $G_i$ with $G_{i-1}$.
    In particular, since $v$ is a valence-one vertex of $G_i$,
    we must have that $\iota_{e_i}\colon \mathcal{G}_{e_i} \to \mathcal{G}_{v}$ is an isomorphism.
    Therefore a valence-two homotopy of $e_j$ across $e_i$ is performable.

    So assume that $H_i$ is irreducible.
    Recall the partial order on strata,
    where $H_i \le H_k$ if some edge in $H_k$ is eventually mapped over some edge in $H_i$
    (and hence any edge in $H_i$, since $H_i$ is irreducible).
    If $H_j$ is a zero stratum, we may after dividing it into two zero strata,
    assume that $H_j = \{e_j\}$.
    This done, if we have $H_i \le H_j$ in this \emph{partial} order,
    and we have that $v$ is an inessential valenece-two vertex
    but $\iota_{e_i} \colon \mathcal{G}_{e_i} \to \mathcal{G}_v$ is not an isomorphism,
    we have a contradiction:
    by assumption the edge $e_j$ is eventually mapped over the edge $e_i$,
    so there is an injective homomorphism $\mathcal{G}_v \cong \mathcal{G}_{e_j} \to \mathcal{G}_{e_i}$,
    which therefore must map $\iota_{e_i}(\mathcal{G}_{e_i})$ properly into itself,
    contradicting the first assumption.

    Finally if we do \emph{not} have $H_i \le H_j$ in this partial order,
    then we may freely move $H_j$ below $H_i$ when we complete the partial order to a total order,
    and thus may swap the roles of $i$ and $j$ if need be.
\end{proof}

The proof of \cite[Lemma 5.4]{BestvinaHandel} shows that
if $i < j$ and $H_i$ is exponentially growing,
then $\pf(f') < \pf(f)$.
In the case where $i = j$ and $H_i$ is exponentially growing,
it may happen that $\lambda_i$ is replaced by some number of eigenvalues $\lambda'$
that all satisfy $\lambda' \le \lambda_i$,
so it is possible that $\pf(f') > \pf(f)$.
Nonetheless, we have the following result.
Call an elementary move \emph{safe} 
if performing it on a topological representative $f\colon\mathcal{G} \to \mathcal{G}$
yields a new topological representative $f'\colon \mathcal{G}' \to \mathcal{G}'$ with $\pf(f') \le \pf(f)$.
Thus all elementary moves with the exception of valence-two homotopy are always safe.

\begin{lem}[\cite{BestvinaHandel} Lemma 5.5]
	\label{boundedlemma}
	If $f\colon\mathcal{G} \to \mathcal{G}$ is a bounded topological representative
	and $f'\colon \mathcal{G}' \to \mathcal{G}'$ is obtained from $f$
	by a sequence of safe moves with
	$\pf(f') < \pf(f)$,
	then there is a bounded topological representative $f''\colon\mathcal{G}'' \to \mathcal{G}''$
	with $\pf(f'') < \pf(f)$.
\end{lem}

\begin{proof}
    The proof is essentially identical to \cite[Lemma 5.5]{BestvinaHandel}.
    Suppose first that our topological representatives satisfy the first assumption.
    By performing valence-one and safe valence-two homotopies,
    we may assume that $f'\colon \mathcal{G}' \to \mathcal{G}'$
    has the property that $\mathcal{G}'$ has no inessential valence-one vertices
    and that each inessential valence-two vertex $v$
    has the property that the two edges incident to $v$ belong to the same exponentially growing stratum.
    Thus $f'\colon \mathcal{G}' \to \mathcal{G}'$ has at most $L$ strata, exponentially growing or no,
    and $\pf(f')$ is obtained from $\pf(f)$ by replacing some of the Perron--Frobenius eigenvalues
    with strictly smaller eigenvalues $\lambda_i'$.
    For the eigenvalues that are not replaced, the fact that $f$ was bounded
    implies that these eigenvalues are the Perron--Frobenius eigenvalues for matrices with at most $L$
    rows and columns.
    Thus we only need to show that the $\lambda_i'$ are also the Perron--Frobenius eigenvalues for matrices
    with at most $L$ rows and columns.
    We do this by performing dangerous valence-two homotopies,
    replacing each $\lambda'_i$ with some collection of $\lambda''_{ij}$ satisfying
    $\lambda''_{ij} \le \lambda'_i$ until each resulting stratum has at most $L$ edges.
    We still have that $f\colon \mathcal{G}'' \to \mathcal{G}''$ has at most $L$ strata,
    so this topological representative $f''\colon \mathcal{G}'' \to \mathcal{G}''$ is bounded.

    In the situation of the second assumption, the argument is essentially the same.
    By \Cref{edgesbound}, there is a uniform bound to the number of problematic valence-two vertices,
    so we need only focus on the inessential valence-two vertices which are \emph{not} problematic.
    We then proceed exactly as above.
\end{proof}

\paragraph{Invariant Core Subdivision.}
We recall the construction of the \emph{invariant core subdivision}
of an exponentially growing stratum $H_r$.
Assume that a topological representative $f\colon \mathcal{G} \to \mathcal{G}$
linearly expands edges over edge paths with respect to some metric on $G$.
If $f(H_r)$ is not entirely contained in $H_r$, then the set
\[
	I_r \coloneqq \{x \in H_r \mid f^k(x) \in H_r \text{ for all } k > 0\}
\]
is an $f$-invariant Cantor set.
The \emph{invariant core} of an edge $e$ in $H_r$ is the smallest closed subinterval of $e$
containing the intersection of $I_r$ with the interior of $e$.
The endpoints of invariant cores of edges in $H_r$ form a finite set which $f$ sends into itself.
Declaring elements of this finite set to be vertices is called \emph{invariant core subdivision.}
The stratum $H_r$ determines a new exponentially growing stratum $H'_r$
whose edges are the invariant cores of edges in $H_r$.

The following lemma says that invariant core subdivision can be used to create topological representatives
whose exponentially growing strata satisfy \hyperlink{EG-i}{(EG-i)}.

\begin{lem}[\cite{BestvinaHandel} Lemma 5.13]
	\label{operation1}
	If $f'\colon\mathcal{G}' \to \mathcal{G}'$
	is obtained from $f\colon\mathcal{G} \to \mathcal{G}$
	by an invariant core subdivision of an exponentially growing stratum $H_r$,
	then $\pf(f') = \pf(f)$, and
	the map $Df'$ maps directions in the resulting exponentially growing stratum $H'_r$ to itself,
	so $H'_r$ satisfies \hyperlink{EG-i}{(EG-i)}.
	If $H_j$ is another exponentially growing stratum for $f\colon\mathcal{G} \to \mathcal{G}$
	that satisfies \hyperlink{EG-i}{(EG-i)} or \hyperlink{EG-ii}{(EG-ii)},
	then the resulting exponentially growing stratum $H'_j$
	for $f'\colon\mathcal{G}' \to \mathcal{G}'$ still satisfies those properties.
\end{lem}

In fact, invariant core subdivision affects only edges in $H_r$.
If new vertices are created, then one or more non-exponentially growing strata are added to the filtration below $H_r$.

\paragraph{Collapsing Inessential Connecting Paths.}
The following lemma says that an application of operations already defined
may be used to construct topological representatives whose exponentially growing strata satisfy
\hyperlink{EG-ii}{(EG-ii)}.

\begin{lem}[\cite{BestvinaHandel} Lemma 5.14]
	\label{operation2}
	Let $f \colon \mathcal{G} \to \mathcal{G}$ be
	a bounded topological representative with exponentially growing stratum $H_r$.
	If $\alpha \subset G_{r-1}$ is a path with endpoints in $H_r \cap G_{r-1}$
	such that $f_\sharp(\alpha)$ is trivial,
	we construct a new bounded topological representative
	$f' \colon \mathcal{G}' \to \mathcal{G}'$
	such that if $H'_r$ is the stratum of $\mathcal{G}'$ determined by $H_r$,
    then either (if the endpoints of $\alpha$ are distinct)
    $H_r'\cap G_{r-1}'$ has fewer points than $H_r\cap G_{r-1}$
    or (if the endpoints of $\alpha$ are equal)
    a vertex group of $H'_r \cap G'_{r-1}$ has increased.

    If $k > r$ and $H_k$ satisfies \hyperlink{EG-ii}{(EG-ii)},
    then $H'_k$, the stratum determined by $H_k$, satisfies \hyperlink{EG-ii}{(EG-ii)}.
    If $k \ge r$ and $H_k$ satisfies \hyperlink{EG-i}{(EG-i)},
    then $H'_k$ satisfies \hyperlink{EG-i}{(EG-i)}.
\end{lem}

\begin{proof}
    We follow the outline of the proof of \cite[Lemma 5.14]{BestvinaHandel}.
    Let $V$ be the vertex set of $G$.
    Subdivide at each point of $\alpha \cap f^{-1}(V)$,
    obtaining a topological representative $f_1 \colon \mathcal{G}(1) \to \mathcal{G}(1)$
    and an identifying homotopy equivalence $p_1\colon \mathcal{G} \to \mathcal{G}(1)$
    whose map of underlying graphs is a homeomorphism but not a cellular map.
    Define $\alpha_1 = p_1(\alpha)$, and write $\alpha_1 = g_0e_1\ldots e_kg_k$.
    There is a map of graphs of groups $h_1\colon \mathcal{G}(1) \to \mathcal{G}$ such that $p_1h_1 = f_1$.
    We may write
    \[  W_1 = h_1(\alpha_1) = (h_1)_{v_0}(g_0)g_{\bar e_1}h_1(e_1)g_{e_1}
    \ldots g_{\bar e_k}h_1(e_k)g_{e_k}^{-1}(h_1)_{v_k}(g_k). \]
    (Here each $h_1(e_i)$ should be understood as the edge determined by the map of underlying graphs.)
    Since $[W_1]$ is trivial,
    there is some backtracking,
    i.e.~there exists $\ell$ such that the edges $h_1(e_\ell)$ and $h_1(\bar e_{\ell+1})$ are equal
    and $g_{e_\ell}(h_1)_{v_\ell}(g_\ell)g_{\bar e_{\ell+1}}$
    belongs to $\iota_{e_\ell}(\mathcal{G}_{e_\ell})$.
    The same statement is true of $f_1$,
    so we may 
    (possibly after twisting the marking or changing the fundamental domain as in \Cref{irreducibleexample})
    fold $e_\ell$ and $\bar e_{\ell+1}$.
    Note that it is possible that $e_\ell = \bar e_{\ell+1}$,
    in which case the fold increases the edge group $\mathcal{G}_{e_\ell}$.
    We get a resulting homotopy equivalence $f_2\colon \mathcal{G}(2) \to \mathcal{G}(2)$
    and the resulting quotient map (which may be a homeomorphism of underlying graphs)
    $p_2\colon \mathcal{G}(1) \to \mathcal{G}(2)$.
    As before, there is a map of graphs of groups $h_2\colon \mathcal{G}(2) \to \mathcal{G}$
    such that now $p_2p_1h_2 = f_2$.
    If the edges $e$ and $e'$ were folded to create an edge $e''$,
    then (thinking of these edges as segments of edges of $\mathcal{G}$)
    we have $f(e) = f(e')$ as length-one edge paths---this is why we twisted the marking---and 
    we define $h_2(e'') = f(e) = f(e')$.
    Define $\alpha_2 = (p_2p_1)_\sharp(\alpha)$, 
    and define $W_2 = h_2(\alpha_2)$ as above.
    We have that $W_2$ is obtained from $W_1$ by canceling some backtracking,
    so $\alpha_2$ has fewer edges than $\alpha_1$.
    We have that $[W_2]$ is trivial, so we may repeat the above argument at most $k$ times
    to produce $f_k\colon \mathcal{G}(k) \to \mathcal{G}(k)$ such that 
    $\alpha_k = (p_kp_{k-1}\dotsb p_2p_1)_\sharp(\alpha)$ is the trivial path.
    Finally let $f'' \colon \mathcal{G}'' \to \mathcal{G}''$
    be the topological representative obtained from $f_k$ 
    by tightening and collapsing the maximal pretrivial forest.
    
    Since folding decreases $\pf(f)$ or leaves it the same, we have $\pf(f'') \le \pf(f)$.
    If $\pf(f'') = \pf(f)$, then $f''$ is bounded since $f$ was, so we let $f' = f''$.
    If not, then we apply \Cref{boundedlemma} to produce a bounded topological representative
    $f' \colon \mathcal{G}' \to \mathcal{G}'$ such that $\pf(f'') \le \pf(f') < \pf(f)$.


    The argument now finishes as in \cite[Lemma 5.14]{BestvinaHandel}.
    If the endpoints of $\alpha$ were distinct,
    then the exponentially growing stratum $H'_r$ determined by $H_r$
    satisfies \[|H'_r \cap G'_{r-1}| < |H_r \cap   G_{r-1}|.\]
    If the endpoints were \emph{not} distinct,
    the vertex group of $H'_r$ determined by the endpoint of $\alpha$
    is now larger than it was in $H_r$,
    in the sense that there is a natural injective but not surjective
    identifying homomorphism.

    As in \cite[Lemma 5.14]{BestvinaHandel}, if $k > r$ and $H_k$ satisfies \hyperlink{EG-ii}{(EG-ii)},
    then the corresponding stratum $H'_k$ of $f'\colon \mathcal{G}' \to \mathcal{G}'$ still satisfies \hyperlink{EG-ii}{(EG-ii)}.
    Likewise, if $k \ge r$ and $H_k$  satisfies \hyperlink{EG-i}{(EG-i)},
    then the corresponding stratum $H'_k$ satisfies \hyperlink{EG-i}{(EG-i)}.
\end{proof}

If $\mathcal{G}$ is a finite graph, 
has finitely generated edge groups,
and has no inessential valence-one vertices,
then the map $f\colon\mathcal{G} \to \mathcal{G}$
satisfies the assumptions of \cite[Theorem 2.1]{Dunwoody},
and thus can be written as a (finite) product of folds and
what Dunwoody calls ``vertex morphisms.''
In fact, the vertex morphisms are unnecessary,
because $f_\sharp$ is an isomorphism.
Since each of the folds performed in \Cref{operation2}
is a fold factor of $f$,
after performing finitely many such folds,
we must have that the exponentially growing stratum of interest $H_r$
satisfies \hyperlink{EG-ii}{(EG-ii)}.

\begin{lem}[\cite{FeighnHandelAlg} Lemma 2.4]
	\label{algorithmcheck}
    Assuming an oracle that can compute products of elements in vertex groups,
    can compute images of injective homomorphisms between edge groups
    and vertex groups of $\mathcal{G}$
    and can tell when two vertex group elements are equal,
	there is an algorithm that checks whether a topological representative
	$f\colon\mathcal{G} \to \mathcal{G}$ is a relative train track map.
\end{lem}

\begin{proof}
	Since \hyperlink{EG-i}{(EG-i)} is a finite property,
    (whether the image of a direction belongs to $H_r$ is a property of the underlying edge,
    and $H_r$ has finitely many edges)
	we may assume that each exponentially growing stratum satisfies
	\hyperlink{EG-i}{(EG-i)}.

	Suppose $H_r$ is an exponentially growing stratum.
	A \emph{connecting path for $H_r$} is a tight path $\alpha$ in $G_{r-1}$
	with endpoints in $H_r \cap G_{r-1}$.
    Since \hyperlink{EG-i}{(EG-i)} holds,
    vertices in $H_r \cap G_{r-1}$ are sent to vertices in $H_r \cap G_{r-1}$.
    For paths with distinct endpoints,
    we claim that for each component $C$ of $G_{r-1}$,
    \hyperlink{EG-ii}{(EG-ii)} for paths with distinct endpoints is equivalent to the condition
    that distinct vertices of $H_r \cap C$ are sent to distinct vertices of $H_r \cap G_{r-1}$.
    Indeed, if this holds, then tight paths with distinct endpoints
    are sent to tight paths with distinct endpoints which are thus homotopically nontrivial.
    If not, then there are a pair of distinct vertices $v$ and $w$
    in $H_r \cap C$ identified by $f$.
    In this case there is a connecting path $\alpha$ with endpoints $v$ and $w$
    whose $f_\sharp$-image is trivial
    (consider what a homotopy inverse does to $f(v') = f(w')$).

    Finally we consider connecting paths with the same endpoint.
    Let $v$ be a vertex in  $H_r \cap C$.
    If the map $f_v\colon \mathcal{G}_v \to \mathcal{G}_{f(v)}$ is an isomorphism 
    there is nothing to check.
    The map $f$ induces an isomorphism 
    $f_\sharp\colon \pi_1(\mathcal{G},v) \to \pi_1(\mathcal{G},f(v))$,
    so we may consider the subgroup $f_\sharp^{-1}(\mathcal{G}_{f(v)})$.
    It is elliptic, and in fact fixes a vertex of the Bass--Serre tree $\Gamma$
    (consider again what a homotopy inverse to $f$ does to the vertex $f(v)$).
    There is a tight path $\sigma$ such that each element of 
    $f_\sharp^{-1}(\mathcal{G}_{f(v)})$
    may be represented by a path of the form $\sigma g\bar\sigma$.
    (This path may not be tight, but may be tightened by a homotopy.)
    Each of these paths is inessential, 
    in the sense that their $f_\sharp$-image is trivial,
    and they are connecting paths for $H_r$ 
    if they are contained in $G_{r-1}$.
    Thus a necessary condition for \hyperlink{EG-ii}{(EG-ii)}
    is that for each such $g \in \mathcal{G}_{f(v)} \setminus f_v(\mathcal{G}_v)$,
    some and hence any tight path homotopic to $\sigma g\bar \sigma$ 
    is \emph{not} contained in $G_{r-1}$.
    In fact this condition is sufficient.
    This is a finite property, since for each vertex $v$ we need only consider the path $\sigma$.
    Therefore we may assume \hyperlink{EG-ii}{(EG-ii)} holds.

	Finally, \hyperlink{EG-iii}{(EG-iii)} for $H_r$ is equivalent to checking that
	$f(e)$ is $r$-legal for each edge $e\in H_r$.
    Since we assume $H_r$ satisfies \hyperlink{EG-i}{(EG-i)},
    in the situation of the first assumption,
    \Cref{coHopfianisomorphic} implies that any nondegenerate turn in $H_r$
    whose directions determine the same underlying oriented edge of $G$ is legal.
    Thus if a turn is illegal, 
    it becomes degenerate as soon as the underlying oriented edges of $G$
    are identified.
    This implies that checking \hyperlink{EG-iii}{(EG-iii)} is a finite property:
    for each of the finitely many turns in $H_r$ crossed by $f(e)$, we need only check
    that either the underlying edges of the turn are periodic,
    so the turn never degenerates,
    or the underlying edges are eventually identified,
    in which case we only need check whether the actual turn degenerates at that stage.
    In the case of the second assumption, there are only finitely many directions
    at a given vertex, so it is clear that the \hyperlink{EG-iii}{(EG-iii)} is a finite property.
\end{proof}

\begin{proof}[Proof of \Cref{relativetraintrackthm}]
    As in the proof of \Cref{traintrackthm},
    we begin with a topological representative $f\colon \mathcal{G} \to \mathcal{G}$
    on a graph of groups satisfying one of our standing assumptions.
    Assume further that $\mathcal{G}$ is reduced.
    By assumption, $f$ is bounded.
	Consider the highest exponentially growing stratum $H_r$ of $\mathcal{G}$.
	We check whether $H_r$ satisfies \hyperlink{EG-i}{(EG-i)} and \hyperlink{EG-ii}{(EG-ii)}
	using \Cref{algorithmcheck}.
	If not, apply \Cref{operation1} and \Cref{operation2} to create a new topological representative,
	still called $f\colon\mathcal{G} \to \mathcal{G}$ 
	such that the resulting exponentially growing stratum $H_r$ satisfies 
	\hyperlink{EG-i}{(EG-i)} and \hyperlink{EG-ii}{(EG-ii)}.
	Repeat with the next highest exponentially growing stratum until all exponentially growing strata
	satisfy these properties.
	Check whether the resulting topological representative, 
	which we still call $f\colon\mathcal{G} \to \mathcal{G}$,
	satisfies \hyperlink{EG-iii}{(EG-iii)}. If it does, we are done.

	If not, then there is some edge $e$ in an exponentially growing stratum $H_r$
	such that $f(e)$ is not $r$-legal.
	We apply the algorithm in the proof of \Cref{traintrackthm}:
	there is a point $P$ in $H_r$ where $f^k$ is not injective at $P$ for some $k > 1$.
	We subdivide and then repeatedly fold.
    As in the proof of \Cref{traintrackthm},
    no edge-group-increasing folds are necessary in this step
    in the case of the first assumption.
    In the contrary case, we have a bound on the number of edge-group-increasing folds.
	Either we have reduced the eigenvalue for $H_r$ or produced a valence-one vertex.
	We remove all valence-one vertices via homotopies
	and perform all possible valence-two homotopies which do not increase $\pf(f)$.
	At this point we have created a new topological representative
	$f'\colon\mathcal{G}' \to \mathcal{G}'$ with $\pf(f') < \pf(f)$,
	but $f'$ may not be bounded.
	Apply \Cref{boundedlemma} to produce a new bounded topological representative
	$f''\colon\mathcal{G}'' \to \mathcal{G}''$ with $\pf(f'') < \pf(f)$.
	If \hyperlink{EG-i}{(EG-i)} and \hyperlink{EG-ii}{(EG-ii)}
	are not satisfied by $f''$, we may restore these properties 
    by applying \Cref{operation1} and \Cref{operation2}.
    We saw that these lemmas preserve boundedness and do not increase $\pf(f'')$.
	Because $\pf(f)$ can only be decreased finitely many times before reaching $\pfmin$,
	eventually this process terminates, yielding a relative train track map.
\end{proof}

\begin{cor}
	\label{pfcorollary}
	If $f\colon\mathcal{G}\to \mathcal{G}$ is a topological representative
	satisfying \hyperlink{EG-i}{(EG-i)} and with $\pf(f) = \pfmin$,
    then $f$ is bounded and
	the exponentially growing strata of $f$ satisfy \hyperlink{EG-iii}{(EG-iii)}.
\end{cor}

\bibliographystyle{alpha}
\bibliography{bib.bib}
\end{document}